\documentclass{amsart}
\usepackage{amssymb,amsmath,amstext,amsthm,amsfonts,amscd}
\usepackage{latexsym}
\usepackage{enumerate}
\usepackage{graphicx}
\newtheorem{theorem}{Theorem}
\newtheorem{proposition}[theorem]{Proposition}
\newtheorem{corollary}[theorem]{Corollary}
\newtheorem{definition}[theorem]{Definition}
\newtheorem{lemma}[theorem]{Lemma}
\newtheorem{remark}[theorem]{Remark}
\newcommand{\Z}{\mathbb{Z}}
\newcommand{\Q}{\mathbb{Q}}
\newcommand{\R}{\mathbb{R}}
\newcommand{\C}{\mathbb{C}}
\newcommand{\N}{\mathbb{N}}
\newcommand{\A}{\mathcal{A}}
\newcommand{\s}{\mathcal{S}}
\newcommand{\m}{\mathrm{m}}
\newcommand{\e}{\mathrm{exp}}
\newcommand{\re}{\mathrm{Re}}
\newcommand{\im}{\mathrm{Im}}
\newcommand{\ie}{{\it{i.$\,$e.\ }}}
\newcommand{\Ok}{E(k)_+}
\newcommand{\Log}{\mathrm{Log}}
\newcommand{\LOG}{\mathrm{LOG}}
\newcommand{\mo}{\mathrm{mod} \ }
\newcommand{\mi}{\C^*\times\R_+^r}
\newcommand{\mip}{\C^*\times\R_+^{r-1}}
\newcommand{\ldeg}{\mathrm{locdeg}}

\begin{document}

\title{Signed Shintani cones for number fields with one complex place}
\author{Milton Espinoza}
\email{milton.espinoza@uv.cl}
\address{Instituto de Matem\'aticas,
Facultad de Ciencias,
   Universidad de Valpara\'iso,
   Gran Breta\~na 1091, 3er piso, Valpara\'iso,
   Chile}
\thanks{This work was partially supported by the Chilean FONDECYT grants 1085153 and 1110277}
\thanks{I would like to address special thanks to my advisor, Eduardo Friedman}

\begin{abstract}
We give a signed fundamental domain for the action on $\C^*\times \R_+^{n-2}$ of the totally positive units $\Ok$ of a number field $k$ of degree $n$ and having exactly one pair of complex embeddings. This signed fundamental domain, built of $k$-rational simplicial cones, is as convenient as a true fundamental domain
for the purpose of studying Dedekind zeta functions. However, while there is no general construction of a true fundamental domain, we construct a
signed fundamental domain from any set of fundamental units of $k$.
\end{abstract}

\maketitle

\tableofcontents

\section{Introduction}
Motivated by the study of special values of $L$-functions over totally real number fields, Shintani introduced in 1976 \cite{Sh1} a geometric method that allowed him to write any partial zeta function of a totally real number field as a finite sum of certain Dirichlet series, which can be considered as a natural generalization of the Hurwitz zeta function.
Later \cite{Sh2} Shintani extended these results to general number fields. In order to enunciate Shintani's geometric method, fix a number field $k$ with $r$
real embeddings and $s$ pairs of complex embeddings (\ie $[k:\Q]=2s+r$), and let $E(k)$ be its group of units. Given a complete set $\tau_i:k\rightarrow\C$ $(1\leq i\leq s+r)$ of embeddings of $k$,
\begin{equation}\label{incrustaciones}
\underbrace{\tau_1,\overline{\tau}_1,\tau_2,\overline{\tau}_2,\dots,\tau_s,\overline{\tau}_s}_{\mathrm{complex \ embeddings}},\underbrace{\tau_{s+1},\tau_{s+2},\dots,\tau_{s+r}}_{\mathrm{real \ embeddings}},
\end{equation}
we can consider $k\subset\C^{s}\times\R^{r}$ by identifying $x\in k$ with
$$\big(x^{(1)},x^{(2)},\dots,x^{(s+r)}\big)\in\C^s\times\R^r,$$
where $x^{(i)}:=\tau_i(x)$. Put
$$\Ok:=E(k)\cap \big(\C^s\times \R_+^{r}\big) \qquad \mathrm{and} \qquad k_+:=k\cap\big((\C^*)^s\times \R_+^{r}\big),$$
where $\R_+^{r}:=(0,\infty)^r$. Then the group $\Ok$ of totally positive units of $k$ acts on $(\C^*)^s\times \R_+^{r}$ by component-wise multiplication, where $(\C^*)^s:=(\C\smallsetminus\{0\})^s$. On the other hand, if $v_1,v_2,\dots,v_d\in \C^s\times\R^r$ ($1\leq d\leq 2s+r$) is a set of
 $\R$-linearly independent vectors, we shall call
$$
C(v_1,v_2,\dots,v_d):=\{t_1v_1+t_2v_2+\dots+t_dv_d \ | \ t_i>0\}
$$
the \emph{$d$-dimensional simplicial cone generated by} $v_1,v_2,\dots,v_d$.

Shintani proved \cite[Proposition 2]{Sh2} that there exists a finite set $\{C_j \ | \ j\in J\}$ of simplicial cones, all with generators in $k_+$ (\ie $k$-rational),
such that
$$ (\C^*)^s\times \R_+^{r} = \bigcup_{j\in J} \ \bigcup_{\varepsilon\in\Ok}\varepsilon C_j \qquad (\mathrm{disjoint \ union}).$$
Equivalently, the finite disjoint union $\bigcup_{j\in J}C_j$ is a fundamental domain for the action of $\Ok$ on
$(\C^*)^s\times \R_+^{r}$. Note that this result does not provide any description of the cones involved.

When $k$ is a totally real number field, Colmez proved \cite{Co1}\cite{Co2} the existence of special units $\eta_1,\eta_2,\dots,\eta_{r-1}\in\Ok$ such that if we put
$$f_{1,\sigma}:=1 \qquad \mathrm{and} \qquad f_{j,\sigma}:=\eta_{\sigma(1)}\eta_{\sigma(2)}\dots\eta_{\sigma(j-1)} \quad (2\leq j\leq r),$$
for $\sigma$ in the symmetric group $S_{r-1}$, then the finite disjoint union
\begin{equation}\label{Colmezcones}
\big\{C_\sigma:=C(f_{1,\sigma},...,f_{r,\sigma}) \ | \ \sigma\in S_{r-1}\big\}
\end{equation}
(together with some boundary faces of the $C_\sigma$) is a fundamental domain of $\R_+^r$ under the action of the group $U$ generated by the $\eta_i$. Unfortunately, we do not know of any practical algorithm for finding these special units when $r\geq 4$.\footnote{See \cite{DF2} for the cubic case $r=3$.}

In 2012, D\'iaz y D\'iaz and Friedman \cite{DF1} removed this obstruction by considering signed fundamental domains. More precisely, if $\eta_1,...,\eta_{r-1}$ is any set of independent units in $\Ok$, then the Colmez cones $C_\sigma$, together with some boundary faces, form a signed fundamental domain for the action on $\R_+^r$ of the group $U$ generated by the $\eta_i$, \ie
\begin{equation}\label{ecuacionvirtual}
\sum_{\substack{w_\sigma=+1\\ \sigma\in S_{r-1}}} \, \sum_{z\in
C_\sigma\cap U \cdot x } w_\sigma\ +\ \sum_{\substack{w_\sigma=-1\\
\sigma\in S_{r-1}}} \,
 \sum_{z\in C_\sigma\cap U \cdot x } w_\sigma \ = \ 1\qquad\big( x\in \R_+^r\big),
\end{equation}
where all sums are over finite sets of cardinality bounded independently of $x$, and $w_\sigma=\pm 1$ is a sign associated to the cone $C_\sigma$.\footnote{In fact, in \cite{Co1}, the special units $\eta_i$
are characterized by the condition $w_\sigma=+1$ for all $\sigma\in S_{r-1}$.} These signed fundamental domains are as convenient as
true fundamental domains for computing partial zeta functions, but they have the advantage of being explicitly constructed from any set of independent units $\eta_1,...,\eta_{r-1}\in\Ok$. To prove their result, D\'iaz y D\'iaz and Friedman used topological degree theory on the quotient manifold $\R_+^r/\Ok$. In the following points we give an overview of such proof, as it helps to understand the present work.

\newcounter{itemcounter0}
\begin{list}
{\textbf{\arabic{itemcounter0}.}}
{\usecounter{itemcounter0}\leftmargin=1.4em}
\item Consider the multiplicative action of $U$ on half-lines $L\subset \R_+^r\cup\{0\}$ with initial point at the origin. Parameterize each $L$ by $y_L\in\R_+^{r-1}$, where $$\{(y_L,1)\}=L\cap\{x\in\R_+^r\,|\,x^{(r)}=1\}.$$
    Then the group $\widetilde{U}:=\langle\widetilde{\eta}_1,\dots,\widetilde{\eta}_{r-1}\rangle$ acts on $\R_+^{r-1}$ by multiplication, where
    $$\widetilde{\eta}_i\in\R_+^{r-1} \quad (1\le i\le r-1), \qquad\qquad \widetilde{\eta}_i^{(j)}:=\eta_i^{(j)}/\eta_i^{(r)} \quad (1\le j\le r-1).$$

\item For each $\sigma\in S_{r-1}$, let $c_\sigma\subset\R_+^{r-1}$ be the set of parameters of half-lines going through the Colmez cone $C_\sigma$ \big(see \eqref{Colmezcones}\big), \ie, $c_\sigma$ is the intersection of $C_\sigma$ with the hyperplane $\{x\in\R_+^r\,|\,x^{(r)}=1\}$. If $\{(c_\sigma,w_\sigma)\}_{\sigma\in S_{r-1}}$ is a signed fundamental domain for the action of $\widetilde{U}$ on $\R_+^{r-1}$, then $\{(C_\sigma,w_\sigma)\}_{S_{r-1}}$ is a signed fundamental domain for the action of $U$ on $\R_+^r$.

\item Let $I^{r-1}:=[0,1]^{r-1}$ be the unit hypercube of $r-1$ dimensions, and consider the usual simplicial decomposition of $I^{r-1}$ into $(r-1)!$ simplices,
    $$I^{r-1}=\bigcup_{\sigma\in S_{r-1}}D_\sigma, \qquad\qquad D_\sigma:=\{y\in I^{r-1} | y^{(\sigma(r-1))}\le\dots\le y^{(\sigma(1))}\}.$$
    There exist two continuous functions $f,f_0:I^{r-1}\to\R_+^{r-1}$ such that
    \begin{enumerate}[(a)]
    \item $f$ is a piecewise affine map that maps $D_\sigma$ onto the closure of $c_\sigma$ for each $\sigma\in S_{r-1}$. The function $f_0$ maps $I^{r-1}$ onto the closure of a fundamental domain for the action of $\widetilde{U}$ on $\R_+^{r-1}$; this fundamental domain is easy to describe but it is not of the form we want.
    \item $f$ and $f_0$ induce homotopic functions $F,F_0:\widehat{T}\to T$ between two tori; $\widehat{T}=I^{r-1}/\!\sim$, with $y\sim y+e_i$ whenever $y,y+e_i\in I^{r-1}$, where $e_i$ is the $i^{\mathrm{th}}$ standard basis vector of $\R^{r-1}$; and $T=\R_+^{r-1}/\widetilde{U}$. Moreover, $F_0$ is a homeomorphism of (global) topological degree $\deg(F_0)=\deg(F)=\pm1$.
    \end{enumerate}

\item Equation \eqref{ecuacionvirtual}, with $C_\sigma$, $U$ and $\R_+^r$ replaced respectively by $c_\sigma$, $\widetilde{U}$ and $\R_+^{r-1}$,
    follows from interpreting the left hand side as a sum of local degrees of $F$ divided by $\deg(F)$ (local-global principle of topological degree theory). Hence $\{(c_\sigma,w_\sigma)\}_{\sigma\in S_{r-1}}$ is a signed fundamental domain for the action of $\widetilde{U}$ on $\R_+^{r-1}$, and the main result of \cite{DF1} follows from point 2.
\end{list}

When $k$ is not totally real, our knowledge of explicit fundamental domains is very limited. There are some examples in a paper \cite{RS} of Sczech and Ren, who found explicit cones
to give numerical evidence of their refinement of Stark's conjecture over complex cubic number fields. A more general approach can be found in
\cite{Ok}, where explicit cones are presented for the field given by the polynomial $X^3+kX-1$. We know of no results for non totally real fields of degree four or more.

The aim of this work is extend the results of \cite{DF1} to number fields $k$ having exactly one complex place. In extending the topological approach of \cite{DF1}, we find two obstructions. The first one is that we have to choose some elements in $k_+$ to generate $[k:\Q]$-dimensional cones together with the given units, unlike the totally real case where the given units provide all the generators for the $r$-dimensional Colmez cones since the rank of the unit group is $r-1$. The other obstruction is that $\mi$ is a non-convex set; this restricts our choice of generators for the cones, which are convex subsets, and also adds considerable technical difficulty to the use of topological degree theory, because there is no obvious way to construct homotopies having the properties we need in a non-convex set. After overcoming these obstructions, our proof will follow the same lines of \cite{DF1} described in the above overview.

To get an idea of our construction, suppose that $k$ is a complex cubic number field, and that $\varepsilon=(\varepsilon^{(1)},\varepsilon^{(2)})\in\C^*\times\R_+$ is a totally positive unit of $k$ of infinite order. Put $\widetilde{\varepsilon}:=\varepsilon^{(1)}/\varepsilon^{(2)}\in\C^*$ and assume $|\widetilde{\varepsilon}|>1$; as in \cite{DF1}, in order to get a signed fundamental domain (built of simplicial cones) for the action of $\langle\varepsilon\rangle$ on $\C^*\times\R_+$, it is sufficient to find a signed fundamental domain (built of triangles) for the action of $\langle\widetilde{\varepsilon}\rangle$ on $\C^*$. For each $\ell=0,1,2$, choose $\alpha_\ell\in\C^*$ such that $\alpha_\ell/|\alpha_\ell|=\e(2\pi i\ell/3)$, and let $\Delta$ be the triangle with vertices $\alpha_0$, $\alpha_1$, $\alpha_2$. We can order the vertices of $\Delta$ and $\widetilde{\varepsilon}\Delta$ by ordering their arguments in $[0,2\pi)$ counterclockwise; of course this depends on $\widetilde{\varepsilon}$. Suppose we get $\alpha_0 < \widetilde{\varepsilon}\alpha_2 < \alpha_1 < \widetilde{\varepsilon}\alpha_0 < \alpha_2 < \widetilde{\varepsilon}\alpha_1;$
if we put
\begin{align*}
&V_1=\{\widetilde{\varepsilon}\alpha_0 , \alpha_2 , \widetilde{\varepsilon}\alpha_1\}, \qquad
&&V_2=\{\alpha_2 , \widetilde{\varepsilon}\alpha_1 , \alpha_0\}, \qquad
&&V_3=\{\widetilde{\varepsilon}\alpha_1 , \alpha_0 , \widetilde{\varepsilon}\alpha_2\}, \\
&V_4=\{\alpha_0 , \widetilde{\varepsilon}\alpha_2 , \alpha_1\}, \qquad
&&V_5=\{\widetilde{\varepsilon}\alpha_2 , \alpha_1 , \widetilde{\varepsilon}\alpha_0\}, \qquad
&&V_6=\{\alpha_1 , \widetilde{\varepsilon}\alpha_0 , \alpha_2\},
\end{align*}
then the triangle $\Delta_\ell$ with vertices $V_\ell$ does not contain the origin for each $\ell=1,\dots,6$, since its vertices lie in a convex subset of $\C^*$. Looking at $\Delta_1$, we deduce that there is a unique $d\in\Z$ such that \ $\arg(\widetilde{\varepsilon})$, \ $2\pi d-2\pi/3$, \ and \ $\arg(\widetilde{\varepsilon})+2\pi/3$ \ lie in an interval of length less than $\pi$, where $\arg(z)$ represents the argument of $z\in\C^*$ in the range $[-\pi,\pi)$. Consider the following elements of $\R^2$:
\begin{align*}
&\phi_{\alpha_2}=(0,d-1/3),  &&\phi_{\alpha_0}=(0,d),  &&\phi_{\alpha_1}=(0,d+1/3), &&\overline{\phi}_{\alpha_2}=(0,d+2/3),\\
&\phi_{\widetilde{\varepsilon}\alpha_0}=(1,0),  &&\phi_{\widetilde{\varepsilon}\alpha_1}=(1,1/3),  &&\phi_{\widetilde{\varepsilon}\alpha_0}=(1,2/3), &&\overline{\phi}_{\widetilde{\varepsilon}\alpha_0}=(1,1);
\end{align*}
also put
\begin{align*}
&V'_1=\{\phi_{\widetilde{\varepsilon}\alpha_0},\phi_{\alpha_2},\phi_{\widetilde{\varepsilon}\alpha_1}\}, \qquad
&&V'_2=\{\phi_{\alpha_2},\phi_{\widetilde{\varepsilon}\alpha_1},\phi_{\alpha_0}\}, \qquad
&&V'_3=\{\phi_{\widetilde{\varepsilon}\alpha_1},\phi_{\alpha_0},\phi_{\widetilde{\varepsilon}\alpha_0}\}, \\
&V'_4=\{\phi_{\alpha_0},\phi_{\widetilde{\varepsilon}\alpha_0},\phi_{\alpha_1}\}, \qquad
&&V'_5=\{\phi_{\widetilde{\varepsilon}\alpha_0},\phi_{\alpha_1},\overline{\phi}_{\widetilde{\varepsilon}\alpha_0}\}, \qquad
&&V'_6=\{\phi_{\alpha_1},\overline{\phi}_{\widetilde{\varepsilon}\alpha_0},\overline{\phi}_{\alpha_2}\};
\end{align*}
and let $\Delta'_\ell$ be the triangle with vertices $V'_\ell$ for each $\ell=1,\dots,6$. If $D$ is the union of all the $\Delta'_\ell$, then $D$ is the closure of a fundamental domain for $\R^2$ under the translation action of its subgroup $\Z^2$, and the $\Delta'_\ell$ form a simplicial decomposition of $D$. Thus we can define a piecewise affine map $f:D\to\C^*$ by
$$f(\overline{\phi}_{\widetilde{\varepsilon}\alpha_0})=\widetilde{\varepsilon}\alpha_0, \qquad f(\overline{\phi}_{\alpha_2})=\alpha_2, \qquad f(\phi_v)=v \quad (\text{for each vertex $v$ of $\Delta$ and $\widetilde{\varepsilon}\Delta$}).$$
Now, the set $\mathcal{F}=\{z\in\C^*\,|\,1\le|z|<|\widetilde{\varepsilon}|\}$ is an obvious fundamental domain for the action of $\langle\widetilde{\varepsilon}\rangle$ on $\C^*$, and the function $f_0:D\to \C^*$ defined by
$$f_0(t,\theta)=\widetilde{\varepsilon}^t\e(2\pi i\theta) \qquad\qquad \big((t,\theta)\in D\big),$$
has image the closure of $\mathcal{F}$; here complex powers are defined by the principal branch of the logarithm. One verifies that $f$ and $f_0$ are homotopic through the homotopy
$$g_\lambda(t,\theta)=\lambda f(t,\theta)+(1-\lambda)f_0(t,\theta) \qquad\qquad (\lambda\in I, \ (t,\theta)\in D);$$
this homotopy is well defined because $f(\Delta'_\ell)$ and $f_0(\Delta'_\ell)$ are contained in a (same) convex subset of $\C^*$ for each $\ell=1,\dots,6$. Furthermore, $f$, $f_0$, and $g_\lambda$ descend to continuous maps between the tori $\widehat{T}=D/\!\sim$ and $T=\C^*/\langle\widetilde{\varepsilon}\rangle$, where $\sim$ identifies points of $D$ lying in the same orbit with respect to the translation action of $\Z^2$ on $\R^2$. This means that the maps between $\widehat{T}$ and $T$ induced by $f$ and $f_0$ are homotopic. From this point forward, our proof follows the same lines of \cite{DF1}. Note that in this case $\alpha_0$, $\alpha_1$, and $\alpha_2$ are not necessarily elements coming from $k_+$; this is a minor problem which will be solved by choosing elements of $k_+$ not ``too far'' from the $\alpha_\ell$.

We are very grateful to the referees for helping us to improve the exposition of this article and for encouraging us to enhance this introduction with an overview of our construction.

\section{The signed fundamental domain}

From now on we assume $r:=\mathrm{rank}\big(\Ok\big)=[k:\Q]-2>0$. Fix a set of independent units $\varepsilon_1,\dots,\varepsilon_r\in\Ok$, and let $V\subset\Ok$ be the subgroup they generate. Following Colmez \cite{Co1}, define
\begin{equation}\label{ftsigma}
f_{t,\sigma} := \varepsilon_{\sigma(1)} \varepsilon_{\sigma(2)}\cdots\;
\varepsilon_{\sigma(t-1)}=\prod_{j=1}^{t-1}
 \varepsilon_{\sigma(j)}\qquad\ (1\le t\le r+1,
 \ \,\sigma\in S_r).
\end{equation}
For $t=1$ we mean $f_{1,\sigma}:=1=(1,1,\dots,1)\in \C^*\times\R^r_+$. Thus $f_{t,\sigma}\in\Ok\subset\C^*\times\R^r_+$. Define
\begin{equation}\label{xitt'sigma}
\xi_\sigma(t,t'):=\tau_1(f_{t,\sigma}^{-1}f_{t',\sigma})\in\C^*  \qquad\qquad (1\le t,t'\le r+1,\ \,\sigma\in S_r),
\end{equation}
where $\tau_1$ is a fixed complex embedding of $k$ \big(see \eqref{incrustaciones}\big). When $t=r+1$ in \eqref{xitt'sigma}, we will write
\begin{equation}\label{xitsigma}
\xi_\sigma(t'):=\xi_\sigma(r+1,t') \qquad\qquad (1\le t'\le r+1,\ \,\sigma\in S_r).
\end{equation}
Note that for all $1\le t,t',t''\le r+1$ and all $\sigma\in S_r$ we have
\begin{equation}\label{xitsigmaprop}
\xi_\sigma(t,t')^{-1}=\xi_\sigma(t',t), \qquad \xi_\sigma(t,t')\cdot\xi_\sigma(t'',t)=\xi_\sigma(t'',t'), \qquad \xi_\sigma(t)\cdot\xi_\sigma(t')^{-1}=\xi_\sigma(t',t).
\end{equation}

Let $\arg(z)$ be the argument in the interval $[-\pi,\pi)$ of the nonzero complex number $z$. For a fixed integer $N\ge3$, let $\m=\m_N:\C^*\rightarrow \Z$ be the function defined by
\begin{equation}\label{m}
\m(z):=\left\lceil\frac{-N\arg(z)}{2\pi} \right\rceil  \qquad (z\in\C^*), \qquad -\frac{N}{2}<\m(z)\leq \left\lceil\frac{N}{2}\right\rceil,
\end{equation}
where the ceiling function $\lceil \ \rceil:\R\to \Z$ satisfies $x\leq\lceil x\rceil<x+1$. Then, for $\sigma\in S_r$ and $t,t'\in\{1,\dots,r+1\}$, consider the next three conditions
\begin{align}
& \m(\xi_\sigma(t,t'))\equiv\m(\xi_\sigma(t'))-\m(\xi_\sigma(t)) \ (\mo N), \label{orden1}\\
& \m(\xi_\sigma(t,t'))+\m(\xi_\sigma(t',t))\equiv1 \ (\mo N), \label{orden2}\\
& t'<t. \label{orden3}
\end{align}
We shall say that $t\prec_\sigma t'$ if and only if the pair $(t,t')$ satisfies condition \eqref{orden1}, and at least one of the conditions \eqref{orden2} and \eqref{orden3}. In Proposition~\ref{ordentotal} we will prove, for $\sigma\in S_r$, that the relation $\prec_\sigma$ is a strict total order on the set $\{1,\dots,r+1\}$. Also, in Lemma~\ref{mprop} we will prove that $\m(\xi_\sigma(t,t'))$ is congruent modulo $N$ to either $\m(\xi_\sigma(t'))-\m(\xi_\sigma(t))$ or $\m(\xi_\sigma(t'))-\m(\xi_\sigma(t))+1$.

Finally, let $\widetilde{S}_r$ be the product of sets
$$S_r\times\{1,\dots,r+1\}\times\{0,\dots,N-1\}$$
with cardinality $\big([k:\Q]-1\big)!\cdot N$.

\subsection{The seven-step algorithm}

With the above conventions and definitions, the following seven steps produce a signed fundamental domain of D\'iaz y D\'iaz--Friedman type \big(see \eqref{ecuacionvirtual}\big) for the action of the group $V$ on $\C^*\times\R^r_+$.

\newcounter{itemcounter}
\begin{list}
{\textbf{\arabic{itemcounter}.}}
{\usecounter{itemcounter}\leftmargin=1.4em}
\item Fix an integer $N\geq 3$, and consider the function $\m=\m_N$ defined in \eqref{m}.

\item For each $\sigma\in S_r$, order the set $\{1,\dots,r+1\}$ using the strict total order $\prec_\sigma$ defined by conditions \eqref{orden1}, \eqref{orden2} and \eqref{orden3}.

\item For each $\sigma\in S_r$, let $\rho_\sigma\in S_{r+1}$ be the unique permutation such that
\begin{equation}\label{rosigma}
\rho_\sigma(r+1)\prec_\sigma\rho_\sigma(r)\prec_\sigma\dots\prec_\sigma\rho_\sigma(2)\prec_\sigma\rho_\sigma(1).
\end{equation}

\item For each $t\in\Z$, choose and fix an element $\alpha_t=\alpha(t)\in k_+$ such that
$$
\alpha_t=\alpha_{t'} \quad \mathrm{if} \quad t\equiv t'(\mo N), \qquad
\arg\!\left(\alpha_t^{(1)}\cdot\e\left(-2\pi i t/N\right)\right)\in\left(\frac{-\pi}{2N},\frac{\pi}{2N}\right).
$$

\item Let $\mu=(\sigma, q, n)\in \widetilde{S}_r$. For $t\in\{1,\dots,r+1\}$, write
\begin{equation}\label{ftsigmaqj}
f_{t,\mu}=f(t,\sigma,q,n):=\begin{cases}
f_{t,\sigma}\cdot\alpha\big(\m(\xi_\sigma(t))+n\big) & \mathrm{if} \ t\nprec_\sigma\rho_\sigma(q), \\
f_{t,\sigma}\cdot\alpha\big(\m(\xi_\sigma(t))+n+1\big) & \mathrm{if} \ t\prec_\sigma\rho_\sigma(q),
\end{cases}
\end{equation}
and for $t=r+2$ write
\begin{equation}\label{ftsigmaqj2}
f_{t,\mu}=f(t,\sigma,q,n):=f_{\rho_\sigma(q),\sigma}\cdot\alpha\!\Big(\m\!\Big(\xi_\sigma\big(\rho_\sigma(q)\big)\Big)+n+1\Big).
\end{equation}

\item For $\mu=(\sigma,q,n)\in \widetilde{S}_r$, define $w_{\mu}=\pm 1$ or $0$ as
\begin{equation}\label{wsigmaqn}
w_{\mu}:=\frac{\mathrm{sgn}(\sigma)
\cdot\mathrm{sign}\big(\!\det(f_{1,\mu}\,,\,f_{2,\mu}\,,\,
\dots\,,\,f_{r+2,\mu})\big)
}{\mathrm{sign}\big(\!\det(\Log\,\,\varepsilon_1,
\Log\,\,\varepsilon_2,\dots,\Log\,\,\varepsilon_r)\big)},
\end{equation}
where $\mathrm{sgn}(\sigma)$ is the usual signature (\ie$\pm1$) of the permutation $\sigma\in S_r$,
$$\Log\,\,\varepsilon_i\in \R^r\qquad \mathrm{with} \qquad\big( \Log\,\,\varepsilon_i  \big)^{(j)}:= \log \, |\varepsilon_i^{(j)}|
\quad \, (1\le j\le r),
$$
the $f_{i,\mu}$ are regarded as elements of $\R^{r+2}$ by the map
\begin{equation}\label{Psi}
(z,x^{(1)},\dots,x^{(r)})\mapsto\big(\re(z),\im(z),x^{(1)},\dots,x^{(r)}\big) \qquad (z\in\C,\,x^{(i)}\in\R),
\end{equation}
and $\,\mathrm{sign}\big(\!\det(v_1,v_2,\dots,v_\ell)\big)$ is the sign of the determinant of the $\ell\times \ell$ real matrix whose columns are the $v_i$.

\item For each $\mu\in \widetilde{S}_r$ with $w_{\mu}\not=0$, consider the real hyperplanes
\begin{equation}\label{hiperplano}
H_{i,\mu}:=\sum_{\substack{1\le t\le r+2\\ t\not=i}}\R\cdot f_{t,\mu} \qquad\qquad (1\le i\le r+2),
\end{equation}
each of which separates $\C\times\R^r$ into two disjoint half-spaces, $\C\times\R^r=H_{i,\mu}^+\cup H_{i,\mu} \cup H_{i,\mu}^-$, where $H_{i,\mu}^+$ is the half-space containing $f_{i,\mu}$. Then define $C_{\mu}=C_\mu(\varepsilon_1,\varepsilon_2,\dots,\varepsilon_r)$ by
\begin{align}\label{Csigmaqn}
\nonumber
C_{\mu}&:=\R_{1,\mu}\cdot f_{1,\mu}+\R_{2,\mu}\cdot f_{2,\mu}+ \cdots+\R_{r+2,\mu}\cdot f_{r+2,\mu},\\
\R_{i,\mu}&:=\begin{cases}
[0,\infty) & \mathrm{if} \ e_{r+2}\in H_{i,\mu}^+, \\
(0,\infty) & \mathrm{if} \ e_{r+2}\in H_{i,\mu}^-,
\end{cases} \qquad\qquad(1\le i\le r+2),
\end{align}
with $e_{r+2}:=[0,0,\dots,0,1]\in\mi$.
\end{list}

\noindent\emph{Some remarks}. The choice $N=3$ in the first step of the algorithm generates the minimum number of cones, namely $([k:\Q]-1)!\cdot3$. Also note that $N$, as well as the $\alpha_t$ chosen in the fourth step, are not included in the posterior notation since they remain fixed along the whole algorithm. In step five, we clearly have $f_{t,\mu}\in k_+\subset\C^*\times\R^r_+$ for all $t\in\{1,\dots,r+2\}$. In step six, note that the absolute value of the determinant in the denominator of \eqref{wsigmaqn} is half of the regulator of the independent units $\varepsilon_1,\varepsilon_2,\dots,\varepsilon_r$, and so is non-zero. Also, in the following when identify $\C\times \R^\ell = \R^{\ell+2}$ as an $\R$-vector space, we will be referring to the isomorphism \eqref{Psi} with $r=\ell$. Finally, the definitions given in the seventh step of the algorithm make sense since if $w_\mu\not=0$, then each closed cone $\overline{C}_{\mu}:=\sum_{t=1}^{r+2}\R_{\ge0}\cdot f_{t,\mu}$ has a non-empty interior; furthermore, in Lemma~\ref{lema49} we will prove that $e_{r+2}$ cannot lie in any of the $H_{i,\mu}$.

\noindent We will call the above algorithm the \emph{seven-step algorithm} (7SA). It produces our main result.

\begin{theorem}\label{Main} Let $k$ be a number field with $r>0$ real embeddings, and exactly one pair of conjugate complex embeddings. Suppose that
the units $\varepsilon_1,\dots,\varepsilon_r$ generate a subgroup $V$ of finite index in the group of totally positive units of $k$. Then the signed cones $\big\{(C_{\mu}, w_{\mu})\big\}_{ w_{\mu}\not=0}$ defined in \eqref{wsigmaqn} and \eqref{Csigmaqn} give a signed fundamental domain for the
action of $V$ on $\C^*\times\R_+^r:=\big(\C\smallsetminus\{0\}\big)\times (0,\infty)^r$. That is,
\begin{equation}\label{Basic}
\sum_{\substack{w_{\mu}=+1\\ \mu\in \widetilde{S}_r}} \, \sum_{z\in
C_{\mu}\cap V \cdot x } w_{\mu}\ +\ \sum_{\substack{w_{\mu}=-1\\ \mu\in \widetilde{S}_r}} \,
 \sum_{z\in C_{\mu}\cap V \cdot x } w_{\mu} \ = \ 1\qquad\big( x\in \mi),
\end{equation}
where all sums are over finite sets of cardinality bounded independently of $x$.
\end{theorem}

\subsection{Corollaries of Theorem 1}

If $w_{\mu}\not=-1$ for all $\mu\in \widetilde{S}_r$, then each orbit $V\cdot x$ must
intersect only one of the $C_{\mu}$'s, and only once at that. Hence

\begin{corollary}
Suppose that \ $w_{\mu}\not=-1$ \ for all \ $\mu\in \widetilde{S}_r$, \ then \ $\cup_{\substack{\mu\in \widetilde{S}_r\\ w_\mu\not=0}} C_{\mu}$ \
is a true fundamental domain for the action of \ $V$ \ on \ $\C^*\times\R_+^r$.
\end{corollary}

The next corollary shows that a signed fundamental domain is as convenient as a true one for dealing with partial zeta functions associated to $k$.\footnote{Its proof coincides with that of \cite[Corollary 6]{DF1}, so we omit it.} Fix an integral ideal $\mathfrak{f}$ of $k$, and put $\mathfrak{f}\infty$ the formal product of $\mathfrak{f}$ with all the infinite places of $k$. Let $\zeta_\mathfrak{f}(\overline{\mathfrak{a}},s):=\sum_{\mathfrak{b}\in\overline{\mathfrak{a}}}\mathrm{N}\mathfrak{b}^{-s}$ ($\re(s)>1$) be the Dedekind partial zeta function attached to a ray class $\overline{\mathfrak{a}}$ modulo $\mathfrak{f}\infty$ represented by the integral ideal $\mathfrak{a}$. Here $\mathfrak{b}$ runs over all integral ideals in $\overline{\mathfrak{a}}$, and $\mathrm{N}$ is the absolute norm.

\begin{corollary}\label{zetaparcial} Suppose $\varepsilon_1,\dots,\varepsilon_r$ generate the group $\Ok^{\mathfrak{f}}$ of totally positive units of $k$ that are congruent to $\mathrm{1}$ mod $\mathfrak{f}$, and suppose we have chosen
$\alpha_0,\dots,\alpha_{N-1}\in\mathfrak{a}^{-1}\mathfrak{f}$ in the fourth step of the 7SA. Then
$$
\zeta_\mathfrak{f}(\overline{\mathfrak{a}},s)=\mathrm{N}\mathfrak{a}^{-s}
\sum_{\substack{\mu\in \widetilde{S}_r\\w_{\mu}\not=0 }} w_{\mu}\sum_{x\in
R_\mathfrak{f}(\mathfrak{a},C_{\mu})}\zeta_\mathfrak{f}(C_{\mu},x,s)\qquad\qquad\big(\re(s)>1\big),
$$
where $\zeta_\mathfrak{f}(C_{\mu},x,s)$ is the Shintani zeta function
\begin{align*}
\zeta_\mathfrak{f}(C_{\mu},x,s)&:=\sum_{n_1,\dots,n_{r+2}=0}^\infty\,  \Big|x^{(1)}+\sum_{t=1}^{r+2} n_tf_{t,\mu}^{(1)}
\Big|^{-2s}\cdot\prod_{j=2}^{r+2} \Big(x^{(j)}+\sum_{t=1}^{r+2} n_tf_{t,\mu}^{(j)}
\Big)^{-s},
\\
R_\mathfrak{f}(\mathfrak{a},C_{\mu}) &:=\Big\{ x\in 1+\mathfrak{a}^{-1}\mathfrak{f} \ \big| \ x=\sum_{t=1}^{r+2} y_t f_{t,\mu},
\ y_t\in I_{t,\mu} \Big\},
\\
I_{t,\mu}&:= \begin{cases}[0,1)\  & \mathrm{if}\ e_{r+2}\in H_{t,\mu}^+ ,\\
  (0,1]\ & \mathrm{if}\ e_{r+2}\in H_{t,\mu}^-.
\end{cases}
\end{align*}
\end{corollary}

\section{Examples}

In this section we show three examples of signed fundamental domains obtained by using the 7SA. Our numerical results are up to an error less than $10^{-28}$.

\subsection{Cubic case} Let $k=\Q(\gamma)$, where $\gamma^3+\gamma^2-1=0$. Then the discriminant of the complex cubic number field $k$ is $-23$. Let
$$\varepsilon_1=\gamma=\frac{1}{\gamma^2+\gamma}=\big[(-0{.}8774...)+(-0{.}7448...)i \, , \, 0{.}7548...\big] \ \in \ \Ok.$$
In the next two examples, we give signed fundamental domains for the action of $\langle \varepsilon_1\rangle$ on $\C^*\times\R_+$.

\subsubsection{Example 1} If $N=3$, $\alpha_0=1$, $\alpha_1=2\gamma^2+2\gamma+1$, and $\alpha_2=2\gamma+1$, one verifies that $\arg\!\big(\alpha_0^{(1)}\big)=0,$
\begin{align*}
\arg\!\big(\alpha_1^{(1)}\cdot\e(-2\pi i/3)\big)=-0{.}2424..., \quad \text{and} \quad
\arg\!\big(\alpha_2^{(1)}\cdot\e(-4\pi i/3)\big)=0{.}0545...
\end{align*}
lie in the interval $(-\pi/6,\pi/6)=(-0{.}5235...,0{.}5235...)$. Following steps 1, 2 and 3, we have
$$\m\!\big(\xi_{(1)}(1)\big)=-1, \qquad \m\!\big(\xi_{(1)}(2)\big)=0,\qquad \m\!\big(\xi_{(1)}(2,1)\big)=-1,\qquad \m\!\big(\xi_{(1)}(1,2)\big)=2,$$
where $(1)\in S_1$ is the identity permutation; hence $2\prec_{(1)}1$, and $\rho_{(1)}$ is the identity permutation of $S_2$. Now, using \eqref{ftsigmaqj} and \eqref{ftsigmaqj2} we compute
\begin{align*}
&f_{1,(1),2,0}=2\gamma+1, && f_{2,(1),2,0}=\gamma, &&f_{3,(1),2,0}=\gamma+2,\\
&f_{1,(1),1,0}=2\gamma+1, &&f_{2,(1),1,0}=\gamma+2, &&f_{3,(1),1,0}=1,\\
&f_{1,(1),2,1}=1, &&f_{2,(1),2,1}=\gamma+2, &&f_{3,(1),2,1}=2\gamma^2+\gamma,\\
&f_{1,(1),1,1}=1, &&f_{2,(1),1,1}=2\gamma^2+\gamma, &&f_{3,(1),1,1}=2\gamma^2+2\gamma+1,\\
&f_{1,(1),2,2}=2\gamma^2+2\gamma+1, &&f_{2,(1),2,2}=2\gamma^2+\gamma, &&f_{3,(1),2,2}=\gamma,\\
&f_{1,(1),1,2}=2\gamma^2+2\gamma+1, &&f_{2,(1),1,2}=\gamma, &&f_{3,(1),1,2}=2\gamma+1.
\end{align*}
Then we compute the $w_\mu$ using \eqref{wsigmaqn}, with $\Log \ \varepsilon_1=0{.}1405...$;
\begin{align*}
&\det\big(f_{1,(1),2,0},f_{2,(1),2,0},f_{3,(1),2,0}\big)=0, &&w_{(1),2,0}=0,\\
&\det\big(f_{1,(1),1,0},f_{2,(1),1,0},f_{3,(1),1,0}\big)=0, &&w_{(1),1,0}=0,\\
&\det\big(f_{1,(1),2,1},f_{2,(1),2,1},f_{3,(1),2,1}\big)=-4{.}7958..., &&w_{(1),2,1}=-1,\\
&\det\big(f_{1,(1),1,1},f_{2,(1),1,1},f_{3,(1),1,1}\big)=4{.}7958..., &&w_{(1),1,1}=+1,\\
&\det\big(f_{1,(1),2,2},f_{2,(1),2,2},f_{3,(1),2,2}\big)=4{.}7958..., &&w_{(1),2,2}=+1,\\
&\det\big(f_{1,(1),1,2},f_{2,(1),1,2},f_{3,(1),1,2}\big)=4{.}7958...., &&w_{(1),1,2}=+1.
\end{align*}
Finally, the following equations allow us to determine the $\R_{i,\mu}$ defined in step 7.
\begin{align*}
e_3&=(-0{.}3681...)f_{1,(1),2,1}+(0{.}3898...)f_{2,(1),2,1}+(0{.}0155...)f_{3,(1),2,1}\\
&=(0{.}0216...)f_{1,(1),1,1}+(-0{.}2344...)f_{2,(1),1,1}+(0{.}3898...)f_{3,(1),1,1}\\
&=(0{.}4114...)f_{1,(1),2,2}+(-0{.}2561...)f_{2,(1),2,2}+(-0{.}0216...)f_{3,(1),2,2}\\
&=(0{.}1553...)f_{1,(1),1,2}+(-0{.}2778...)f_{2,(1),1,2}+(0{.}2561...)f_{3,(1),1,2}.
\end{align*}
Therefore, the cones of the signed fundamental domain are
\begin{align*}
&C_{(1),2,1}=\{t_1+t_2(\gamma+2)+t_3(2\gamma^2+\gamma) \ | \ t_1>0, \ t_2\geq 0, \ t_3\geq 0\},\\
&C_{(1),1,1}=\{t_1+t_2(2\gamma^2+\gamma)+t_3(2\gamma^2+2\gamma+1) \ | \ t_1\geq 0, \ t_2> 0, \ t_3\geq 0\},\\
&C_{(1),2,2}=\{t_1(2\gamma^2+2\gamma+1)+t_2(2\gamma^2+\gamma)+t_3\gamma \ | \ \ t_1\geq 0, \ t_2>0, \ t_3>0\},\\
&C_{(1),1,2}=\{t_1(2\gamma^2+2\gamma+1)+t_2\gamma+t_3(2\gamma+1) \ | \ \ t_1\geq 0, \ t_2> 0, \ t_3\geq 0\}.
\end{align*}
Figure~\ref{Figure1} represents the intersection of the plane $\{(z,1) \ | \ z\in\C \}\subset \C\times\R$ with the signed fundamental domain. The blue region indicates the cones with $w_\mu$ positive, the red region indicates the cone with $w_\mu$ negative, and the purple region represents the intersection of two cones with opposite signs.

\begin{figure}
\includegraphics[scale=0.75]{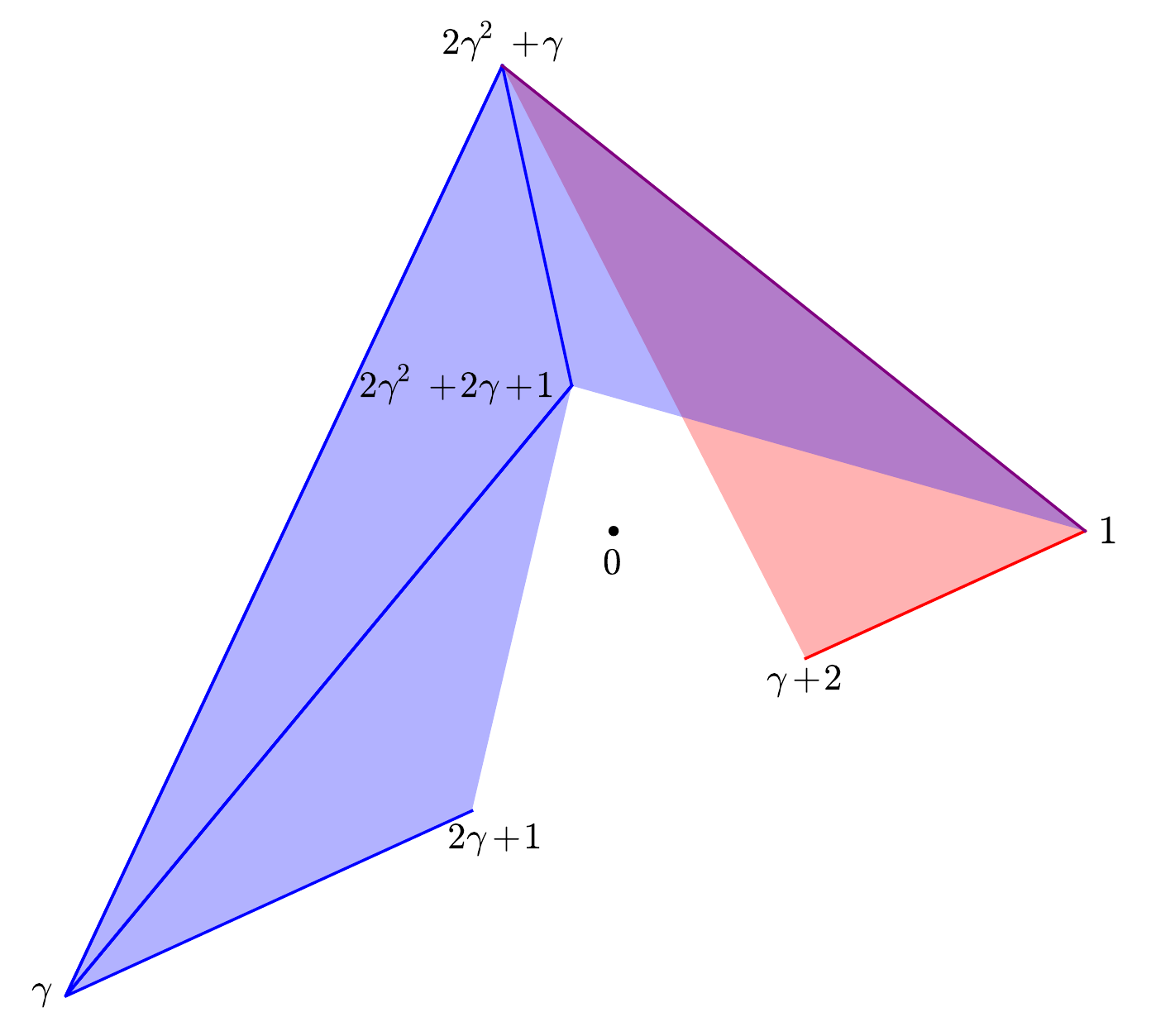}
\caption{Signed fundamental domain for the action of $\langle\gamma\rangle$ on $\C^*\times\R_+$, where $\gamma^3+\gamma^2-1=0$, $\alpha_0=1$, $\alpha_1=2\gamma^2+2\gamma+1$, and
  $\alpha_2=2\gamma+1$.}
\label{Figure1}
\end{figure}

\subsubsection{Example 2} If $N=3$, $\alpha_0=1$, $\alpha_1=\gamma^2+\gamma$, and $\alpha_2=\gamma$, then the 7SA gives
$$C_{(1),1,1}=\{t_1+t_2\gamma^2+t_3(\gamma^2+\gamma) \ | \ \ t_1\geq 0, \ t_2> 0, \ t_3\geq 0\},$$
$w_{(1),1,1}=+1$, and $w_\mu=0$ for all $\mu\in \widetilde{S}_r$ with $\mu\not=\big((1),1,1\big)$. Therefore, in this case the 7SA gives a true fundamental domain for the action of $\langle\gamma\rangle$ on $\C^*\times\R_+$.

\subsection{Quartic case} Let $k=\Q(\gamma)$, where $\gamma^4+\gamma-1=0$. Then the discriminant of $k$ is $-283$. Let
$$\varepsilon_1=\gamma^2=\frac{1}{\gamma^3+\gamma^2+1} \qquad \mathrm{and} \qquad \varepsilon_2=\gamma^2+1=\frac{1}{\gamma^3-\gamma+1}$$
be two independent totally positive units of $k$, with
$$\gamma=\big[(0{.}2481...)+(-1{.}0339...)i \, , \, -1{.}2207... \, , \, 0{.}7244...\big] \ \in \ \C\times\R^2.$$

\subsubsection{Example 3} If $N=3$, $\alpha_0=1$, $\alpha_1=\gamma^2-\gamma+1$, and $\alpha_2=\gamma^2+\gamma$, then the 7SA gives the signed fundamental domain for the action of $\langle\varepsilon_1,\varepsilon_2\rangle$ on $\C^*\times\R_+^2$ with
\begin{align*}
C_{(1),2,0}=\{&t_1(\gamma^2-\gamma+1)+t_2\gamma^2+t_3(-2\gamma^3+3\gamma^2-3\gamma+2)+t_4(\gamma^2+\gamma) \ |\\
&t_1\ge0, \ t_2>0, \ t_3>0, \ t_4\ge0\},\\
C_{(1),3,1}=\{&t_1(\gamma^2+\gamma)+t_2(-\gamma^3+\gamma^2-\gamma+1)+t_3(-2\gamma^3+3\gamma^2-3\gamma+2)+t_4 \ |\\
&t_1\ge0, \ t_2>0, \ t_3\ge0, \ t_4\ge0\},\\
C_{(12),1,0}=\{&t_1(\gamma^2+\gamma)+t_2(\gamma^3+\gamma^2+1)+t_3(-2\gamma^3+3\gamma^2-3\gamma+2)+t_4(\gamma^2+1) \ |\\
&t_1>0, \ t_2\ge0, \ t_3>0, \ t_4\ge0\},\\
C_{(12),3,1}=\{&t_1(\gamma^2+\gamma)+t_2(\gamma^2+1)+t_3(-2\gamma^3+3\gamma^2-3\gamma+2)+t_4 \ |\\
&t_1>0, \ t_2\ge0, \ t_3>0, \ t_4>0\},
\end{align*}
and $w_{(1),2,0}=w_{(1),3,1}=w_{(12),1,0}=w_{(12),3,1}=+1$. The rest of the $w_\mu$ are 0. So, as in the previous example, this signed fundamental domain is actually a true one.

\section{Construction of $f$} \label{constructionoff}

As in \cite{DF1}, we will prove Theorem~\ref{Main} by interpreting the left-hand side of \eqref{Basic} as a sum of local degrees of a certain continuous map $F:\widehat{T}\to T$ between a standard $(r+1)$-torus $\widehat{T}$ and an $(r+1)$-torus $T$. Using a basic result in algebraic topology, this sum of local degrees equals the global degree of $F$. We will compute this global degree by proving that $F$ is homotopic to an explicit homeomorphism $F_0$, whose degree can be easily computed. Our contribution lies in the construction of a piecewise affine map $f$, which we will use to define $F$.

\subsection{The argument at the complex embedding}

As we said in the Introduction, the non-convexity of $\mi$ is an obstruction to deal with. To bypass this obstruction, we will divide $\mi$ into certain convex regions using the argument at the complex place. For $N\in\N$ (with
$N\geq 3$ and fixed), we define the regions
\begin{equation}\label{semiplanos}
\s_t=\s_{t,N}:=\e\left(2\pi it /N\right)\cdot\s_0\subset\C^* \qquad\qquad (t\in\Z),
\end{equation}
where
$$
\s_0=\s_{0,N}:=\left\{z\in\C^* \ \left| \ \arg(z)\in [-\pi/2N,5\pi/2N) \right.\right\}.
$$
Since $N\ge3$ the $\s_t$ are convex, and their union for $t\in\Z$ is $\C^*$. Also $\s_t=\s_{t'}$ if and only if $t\equiv t'(\mo N)$.

Now we define the ``windmill arms'' $\A_t$ by
\begin{equation}\label{aspas}
\A_t=\A_{t,N}:=\e\left(2\pi it/N\right)\cdot\A_0 \qquad\qquad (t\in\Z),
\end{equation}
where
$$\A_0=\A_{0,N}:=\left\{z\in\C^* \ \left| \ \arg(z)\in [-\pi/2N,\pi/2N] \right.\right\}.$$
Since $\A_0$ and the interior $\stackrel{\circ}{\A}_1$ of $\A_1$ are contained in $\s_0$, we have
\begin{equation}\label{aspasemiplano}
\A_t\subset\s_t, \qquad\qquad \stackrel{\circ}{\A}_{1+t}\subset\s_t, \qquad\qquad \A_{1+t}\not\subset\s_t\qquad\qquad (t\in\Z).
\end{equation}
Figure~\ref{Figure2} shows the windmill arms $\mathcal{A}_t$ in the case $N=3$.

\begin{figure}
\includegraphics[scale=0.5]{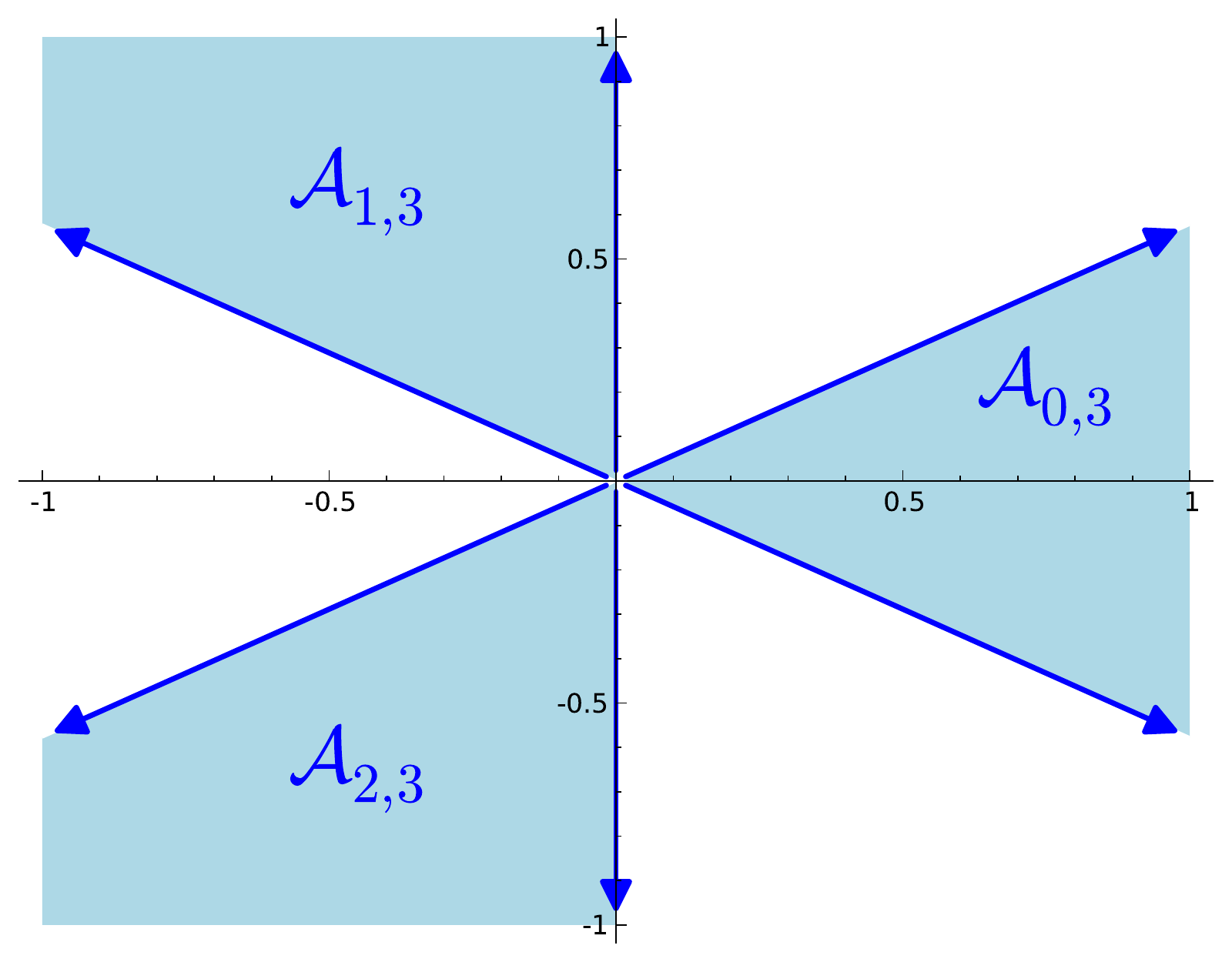}
\caption{$\mathcal{A}_t$ for $N=3$.}
\label{Figure2}
\end{figure}

Before continuing the study of the regions described above, we need some elementary properties of the function $\m:\C^*\to\Z$ defined in \eqref{m}. In the following, all the congruences ($\equiv$) will be modulo $N$.

\begin{lemma} \label{mprop}
Let $z,u,v,w\in\C^*$. Then the following hold.
\begin{enumerate}[$(i)$]
\item $\m(zw)$ \ is congruent to either \ $\m(z)+\m(w)$ \ or \ $\m(z)+\m(w)-1$.

\item If \ $\m(z)+\m(z^{-1})\equiv0$, \ then \ $\m(zw)\equiv\m(z)+\m(w)$.

\item $\arg\!\Big(z\cdot\e\big(2\pi i\m(z)/N\big)\Big)$ \ lies in \ $\big[0,\,2\pi/N\big)$.

\item If \ $\m(u^{-1}v)+\m(v^{-1}w)\equiv\m(w)-\m(u)$, \ then \ $\m(u^{-1}w)\equiv\m(w)-\m(u)$.

\item We have \ $\m(v^{-1}u)+\m(vu^{-1})\equiv0$ \ if the following four equations hold;
\begin{align*}
&\m(u^{-1}w)+\m(uw^{-1})\equiv0, &&\m(u^{-1}v)\equiv\m(v)-\m(u),\\
&\m(u^{-1}w)\equiv\m(w)-\m(u), &&\m(v^{-1}w)\equiv\m(w)-\m(v).
\end{align*}
\end{enumerate}
\end{lemma}

\begin{proof}
First note that for all $x,y\in\R$ and all $\ell'\in\Z$ we have
$$
\lceil x+\ell'\rceil=\lceil x\rceil+\ell' \qquad \mathrm{and} \qquad \lceil x\rceil + \lceil y\rceil -1\leq \lceil x+y\rceil\leq \lceil x\rceil + \lceil y\rceil.
$$
Thus (i) follows easily from \eqref{m}, and from these two properties of the ceiling function.

To prove (ii), first note that $\m(zw)$ is congruent to either $\m(z)+\m(w)-1$ or $\m(z)+\m(w)$ \big(by (i)\big). Suppose $\m(zw)\equiv\m(z)+\m(w)-1$. Using (i), we have that $\m(w)=\m(zwz^{-1})$ is congruent to either $\m(zw)-\m(z)$ or $\m(zw)-\m(z)-1$ \big(since $-\m(z)\equiv\m(z^{-1})$\big), and so congruent to either $\m(w)-1$ or $\m(w)-2$, which is absurd since $N\ge3$.

Let us prove (iii). Using the identity $\lceil x+\ell'\rceil=\lceil x\rceil+\ell'$ $(x\in\R, \ \ell'\in\Z)$, we have that $\m\!\Big(z\cdot\e\big(2\pi i\m(z)/N\big)\Big)$ is congruent to $\left\lceil\frac{-N}{2\pi}\big(\arg(z)+2\pi\m(z)/N\big)\right\rceil$, and so congruent to 0.
In general, if $w'\in\C^*$ is such that $\m(w')\equiv 0$, then we have that $\lceil-N\arg(w')/2\pi-Nq\rceil=0$
for some $q\in\Z$. But this is equivalent to $0\leq\arg(w')+2\pi q<2\pi/N$, so $q=0$. Therefore, we have proved (iii).

If $\m(u^{-1}v)+\m(v^{-1}w)+\m(u)$ is congruent to $\m(u^{-1}w)+1+\m(u)$, then it is congruent to either $\m(w)+2$ or $\m(w)+1$ by (i), but this is absurd since $N\ge 3$. Hence, $\m(u^{-1}w)+\m(u)\equiv\m(w)$.

To prove (v), suppose $\m(v^{-1}u)+\m(vu^{-1})\equiv1$. Using (ii) we have that
$$\m(wv^{-1})\equiv\m(u^{-1}w\cdot v^{-1}u)\equiv\m(u^{-1}w)+\m(v^{-1}u),$$
but while the left-hand side of these congruences is congruent to $\m(w)-\m(v)$, the right-hand side is congruent to $\m(w)-\m(v)+1$, which is absurd. Therefore, from (i) we have $\m(v^{-1}u)+\m(vu^{-1})\equiv0$, since $\m(1)=0$.
\end{proof}

Next we give necessary and sufficient conditions for some inclusion relations of the regions $\A_t$ and $\s_t$. These conditions are based on modular arithmetic, and they allow us to relate $\A_t$ and $\s_t$ with the relation $\prec_\sigma$ defined by conditions \eqref{orden1}, \eqref{orden2} and \eqref{orden3}.

\begin{lemma} \label{aspadentrosemiplano}
Let $z\in\C^*$ and let $t,k\in\Z$. Then the following hold.
\begin{enumerate}[$(i)$]
\item $z\cdot\A_t\subset\s_k$ \ if and only if \ $\m(z)\equiv t-k$.

\item $z\cdot\A_t=\A_k$ \ if and only if \ $\m(z)+\m(z^{-1})\equiv0$ \ and \ $\m(z)\equiv t-k$.
\end{enumerate}
\end{lemma}

\begin{proof}
To prove (i), first assume $z\cdot\A_t\subset\s_k$. From \eqref{aspas} it is clear that $z\cdot\A_t=\e\big(i\arg(z)\big)\cdot\A_t$. Thus, from
\eqref{semiplanos} we get
$$\e\big(i\arg(z)\big)\cdot\e(2\pi it/N)\cdot\e(-2\pi ik/N)\cdot w' \ \in \ \s_0$$
for all $w'\in\A_0$. Putting $w'=\e(-\pi i/2N)$ and then $w'=\e(\pi i/2N)$, we see that there exist $q,q'\in\Z$ such that
\begin{align}
\nonumber
&-\pi/2N\leq\arg(z)+2\pi t/N-2\pi k/N-\pi/2N+2\pi q<5\pi/2N,\\
\nonumber
&-\pi/2N\leq\arg(z)+2\pi t/N-2\pi k/N+\pi/2N+2\pi q'<5\pi/2N.
\end{align}
This implies that
\begin{align}
\nonumber
&-3/2<-N\arg(z)/2\pi-t+k-Nq\leq0,\\
\nonumber
&-1<-N\arg(z)/2\pi-t+k-Nq'\leq 1/2.
\end{align}
If $t-k\not\equiv \m(z)$, the above would imply that $\m(z)-t+k$ is congruent to both $\pm 1$, which is absurd since $N\geq 3$. Conversely, suppose $t-k\equiv \m(z)$. Let $w'\in\A_0$. Using Lemma~\ref{mprop}~(iii) and $\arg(w')\in[-\pi/2N,\pi/2N]$, we have for some $t'\in\Z$
$$\arg\!\Big(z\cdot\e\big(2\pi i\m(z)/N\big)\cdot w'\Big)+2\pi t' \ \in \ [-\pi/2N,5\pi/2N) \ \subset \ [-\pi,\pi),$$
so $t'=0$. Then, using \eqref{semiplanos}, \eqref{aspasemiplano}, and that $t\equiv \m(z)+k$, we get that
\begin{align*}
z\cdot\A_t=z\cdot\e\left(2\pi it/N\right)\cdot\A_0=z\cdot\e\big(2\pi i\m(z)/N\big)\cdot\e(2\pi ik/N)\cdot\A_0
\end{align*}
is contained in $\e(2\pi ik/N)\cdot\s_0=\s_k$. This concludes the proof of (i).

From \eqref{aspas}, we have $z\cdot\A_t=\A_k$ if and only if $\arg(z)+2\pi(t-k)/N=2\pi q$ for some $q\in\Z$. So if $z\cdot\A_t=\A_k$, then $$-N\arg(z)/2\pi=t-k+Nq \qquad \mathrm{and} \qquad -N\arg(z^{-1})/2\pi=k-t+Nq'$$
for some $q,q'\in\Z$. Hence, $\m(z)+\m(z^{-1})\equiv0$, and $\m(z)\equiv t-k$, using definition \eqref{m}. Conversely, suppose $\m(z)+\m(z^{-1})\equiv0$, and $\m(z)\equiv t-k$.  Since $\lceil x\rceil+\lceil -x\rceil$ equals either 0 or 1 depending on whether $x\in\Z$ or $x\in\R\smallsetminus\Z$ respectively, we see that $\m(z)+\m(z^{-1})\equiv0$ implies $-N\arg(z)/2\pi\in\Z$. Hence, $\m(z)\equiv t-k$ implies $\arg(z)+2\pi(t-k)/N=2\pi q$ for some $q\in\Z$.
\end{proof}

Let $\sigma\in S_r$. For $(t,t')\in\Z\times\Z$ with $1\leq t,t'\leq r+1$, consider:
\begin{align}
\xi_\sigma(t,t')\cdot\A_{\m(\xi_\sigma(t'))}&\subset \s_{\m(\xi_\sigma(t))},\label{1.2}\\
\xi_\sigma(t,t')\cdot\A_{\m(\xi_\sigma(t'))}&\not=\A_{\m(\xi_\sigma(t))},\label{1.1}\\
t'&<t.\label{2.2}
\end{align}
Using Lemma~\ref{aspadentrosemiplano}, and the definition of $\prec_\sigma$ (see conditions \eqref{orden1}, \eqref{orden2} and \eqref{orden3}), we get that $t\prec_{\sigma}t'$ if and only if $(t,t')$ satisfies condition \eqref{1.2}, and at least one of the conditions \eqref{1.1} and \eqref{2.2}. Now we prove that $\prec_\sigma$ is a strict total order on the set $\{\ell\in\Z ; 1\leq \ell\leq r+1\}$.

\begin{proposition} \label{ordentotal}
For each $\sigma\in S_r$, the relation $\prec_\sigma$ is a strict total order on the set $\{\ell\in\Z ; 1\leq \ell\leq r+1\}$.
\end{proposition}

\begin{proof} As $\sigma$ remains fixed along the proof, we will exclude it from the notation; furthermore, we will write
$M(t,t'):=\m(\xi_\sigma(t,t'))$ and $M(t):=\m(\xi_\sigma(t))$ for any $1\le t, t'\le r+1$.

\noindent\emph{Transitivity}. Suppose $t\prec t'$ and $t'\prec t''$. Using condition \eqref{orden1}, we have
$$M(t,t')+M(t',t'')\equiv M(t'')-M(t).$$
Then putting $u=\xi(t)$, $v=\xi(t')$, $w=\xi(t'')$ in Lemma~\ref{mprop}~(iv), and using \eqref{xitsigmaprop}, the last congruence implies condition \eqref{orden1} for $(t,t'')$.

Now suppose $(t,t'')$ satisfies neither \eqref{orden2} nor \eqref{orden3}. From (i) and \eqref{xitsigmaprop} we have $M(t,t'')+M(t'',t)\equiv0$. Also, we have that the pairs $(t,t')$, $(t',t'')$, and $(t,t'')$ satisfy condition \eqref{orden1}. Thus we have satisfied the hypotheses of Lemma~\ref{mprop}~(v) with $u=\xi(t)$, $v=\xi(t')$, and $w=\xi(t'')$, and also with $u=\xi(t'')$, $v=\xi(t')$, and $w=\xi(t)$. Then $M(t,t')+M(t',t)\equiv0$ and $M(t',t'')+M(t'',t')\equiv0$, so $t'<t$ and $t''<t'$, which contradicts $t''\ge t$. Therefore, $(t,t'')$ must satisfy at least one of the conditions \eqref{orden2} and \eqref{orden3}.

\noindent\emph{Trichotomy law}. Suppose $t'\not=t$ (say $t'<t$). If $(t,t')$ does not satisfy condition \eqref{orden1}, then $M(t,t')$ is congruent to $M(t')-M(t)+1$ by using Lemma~\ref{mprop}~(i) with $z=\xi(t,t')$ and $w=\xi(t)$. Also we have $M(t,t')+M(t',t)\equiv1$ by using Lemma~\ref{mprop}~(ii) with $z=\xi(t,t')$ and $w=\xi(t)$. Combining these two congruences we get that $(t',t)$ satisfies conditions \eqref{orden1} and \eqref{orden2}, so $t'\prec t$. If $(t,t')$ satisfies \eqref{orden1}, then $t\prec t'$.

Now if $t\prec t'$ and $t'\prec t$, condition \eqref{orden1} for $(t,t')$ and $(t',t)$ implies that the pairs $(t,t')$ and $(t',t)$ do not satisfy \eqref{orden2}, so $t<t'$ and $t'<t$, which is absurd. Also, it is clearly impossible that $t=t'$ and $t\prec t'$.
\end{proof}

\begin{corollary}\label{Cor0.5}
Let \ $\sigma\in S_r$, \ and let \ $(t, t')\in\Z\times\Z$ \ with \ $1\leq t, t'\leq r+1$. If \ $t\prec_\sigma t'$, \ then \ $\xi_\sigma(t',t)\cdot\stackrel{\circ}{\A}_{1+\m(\xi_\sigma(t))}\subset \s_{\m(\xi_\sigma(t'))}$.
\end{corollary}

\begin{proof}
Again $\sigma$ remains fixed along the proof, so we will use the notation adopted in the proof of Proposition~\ref{ordentotal}.

Using \eqref{xitsigmaprop} and Lemma~\ref{mprop}~(i), we have that $M(t',t)$ is congruent to either $M(t)-M(t')$ or $M(t)-M(t')+1$. If $M(t',t)\equiv M(t)-M(t')$, then $M(t',t)+M(t,t')\equiv0$ by trichotomy. Hence, Lemma~\ref{aspadentrosemiplano}~(ii) implies that $\xi(t',t)\cdot\A_{1+M(t)}=\A_{1+M(t')}$, and so
$$\xi(t',t)\cdot\stackrel{\circ}{\A}_{1+M(t)}=\stackrel{\circ}{\A}_{1+M(t')}\subset
\s_{M(t')}$$
by \eqref{aspasemiplano}. On the other hand, if $M(t',t)\equiv M(t)-M(t')+1$, then Lemma~\ref{aspadentrosemiplano}~(i) implies that $\xi(t',t)\cdot\A_{1+M(t)}\subset \s_{M(t')}$.
\end{proof}

The order $\prec_\sigma$ depends on the permutation
$\sigma\in S_r$ by definition. In general, we are not interested in studying the behavior of $\prec_\sigma$ with respect to $\sigma$, except in the following case.

\begin{lemma}\label{lema42}
Let \ $\sigma\in S_r$. \ Define \ $\widetilde{\sigma}\in S_r$ \ by putting \ $\widetilde{\sigma}(1):=\sigma(r)$, \ and
\ $\widetilde{\sigma}(j):=\sigma(j-1)$ \ for each \ $2\leq j\leq r$. \ Consider the set
$$B_\sigma:=\left\{1\leq t\leq r+1 \ \left| \ \m\big(\varepsilon_{\sigma(r)}^{(1)}\cdot\xi_\sigma(t)\big) \equiv
\m\big(\varepsilon_{\sigma(r)}^{(1)}\big)+\m(\xi_\sigma(t))\right.\right\}$$
and its complement \ $B_\sigma^c\subset\{1,\dots,r+1\}$. \ Then for any \ $t,t'\in\{1,\dots,r\}$ \ we have
\begin{enumerate}[$(i)$]
\item
$f_{t+1,\widetilde{\sigma}}=\varepsilon_{\sigma(r)}\cdot f_{t,\sigma}$, \qquad
    $\xi_{\widetilde{\sigma}}(t+1)=\varepsilon_{\sigma(r)}^{(1)}\cdot \xi_\sigma(t)$, \qquad $\xi_{\widetilde{\sigma}}(t+1,t'+1)=\xi_\sigma(t,t').$

\item If \ $t,t'\in B_\sigma$ \ or \ $t,t'\in B_\sigma^c$, \ then \ $t\prec_\sigma t'$ \ if and only if \ $t+1\prec_{\widetilde{\sigma}}t'+1$.

\item If \ $t\in B_\sigma$ \ and \ $t'\in
    B_\sigma^c$, \ then \ $t\prec_\sigma t'$ \ and \ $t'+1\prec_{\widetilde{\sigma}}t+1$.
\end{enumerate}
\end{lemma}

\begin{proof}
Note that since $\xi_\sigma(r+1)=1$ and $\m(1)=0$, we have that $r+1\in B_\sigma$ for all $\sigma\in S_r$, so $B_\sigma\not=\varnothing$. The fact that $f_{t+1,\widetilde{\sigma}}=\varepsilon_{\sigma(r)}\cdot f_{t,\sigma}$ follows easily from the definition \eqref{ftsigma} of $f_{t,\sigma}$, and from the definition of $\widetilde{\sigma}$. Moreover, since $f_{r+1,\sigma}$ does not depend on $\sigma$, we obtain (i) from the definition \eqref{xitsigma} of $\xi_\sigma(t)$, and from \eqref{xitsigmaprop}.

Let us prove (ii). Using (i), we have $\m(\xi_{\widetilde{\sigma}}(t+1,t'+1))\equiv\m(\xi_\sigma(t,t'))$, so it is clear that $(t+1,t'+1)$ satisfies \eqref{orden2} for $\widetilde{\sigma}$ if and only if $(t,t')$ satisfies \eqref{orden2} for $\sigma$. But $(t+1,t'+1)$ satisfies \eqref{orden3} for $\widetilde{\sigma}$ if and only if $(t,t')$ satisfies \eqref{orden3} for $\sigma$. Now if $t,t'\in B_\sigma$, we obtain $$\m(\xi_{\widetilde{\sigma}}(t'+1))-\m(\xi_{\widetilde{\sigma}}(t+1))\equiv\m(\xi_\sigma(t'))-\m(\xi_\sigma(t))$$
by using (i), so we have that $(t+1,t'+1)$ satisfies \eqref{orden1} for $\widetilde{\sigma}$ if and only if $(t,t')$ satisfies \eqref{orden1} for $\sigma$.
If $t,t'\in B_\sigma^c$, the proof follows analogously, noting that
\begin{equation}\label{Bsigmac}
\ell\in B_\sigma^c \ \Longleftrightarrow \ \m\big(\xi_{\widetilde{\sigma}}(\ell+1)\big) \equiv
\m\big(\varepsilon_{\sigma(r)}^{(1)}\big)+\m(\xi_\sigma(\ell)) -1.
\end{equation}

To prove (iii), first we will prove $t\prec_\sigma t'$ by contradiction; suppose $t'\prec_\sigma t$. Since $t\in B_\sigma$ and $t'\in B_\sigma^c$ we have
$$\m(\xi_{\widetilde{\sigma}}(t+1))-\m(\xi_{\widetilde{\sigma}}(t'+1))\equiv\m(\xi_\sigma(t))-\m(\xi_\sigma(t'))+1,$$
where the right-hand side is congruent to $\m(\xi_\sigma(t',t))+1$ by using that $(t',t)$ satisfies condition \eqref{orden1} for $\sigma$, and so congruent to $\m(\xi_{\widetilde{\sigma}}(t'+1,t+1))+1$ by using (i). But this contradicts Lemma~\ref{mprop}~(i) since $N\ge3$.

Finally, we will prove that $t\prec_\sigma t'$ implies $t'+1\prec_{\widetilde{\sigma}}t+1$. For the sake of contradiction, suppose that $t+1\prec_{\widetilde{\sigma}} t'+1$ and $t\prec_\sigma t'$. Thus condition \eqref{orden1} implies that $\m(\xi_{\widetilde{\sigma}}(t'+1))-\m(\xi_{\widetilde{\sigma}}(t+1))$ is congruent to $\m(\xi_{\widetilde{\sigma}}(t+1,t'+1))$, and so congruent to $\m(\xi_\sigma(t,t'))$ by (i). On the other hand, since $t\in B_\sigma$ and $t'\in B_\sigma^c$ we have
$$\m(\xi_{\widetilde{\sigma}}(t+1))-\m(\xi_{\widetilde{\sigma}}(t'+1))\equiv\m(\xi_\sigma(t))-\m(\xi_\sigma(t'))+1,$$
where the right-hand side is congruent to $1-\m(\xi_\sigma(t,t'))$ since $(t,t')$ satisfies condition \eqref{orden1} for $\sigma$. Hence
$$\m(\xi_{\widetilde{\sigma}}(t+1))-\m(\xi_{\widetilde{\sigma}}(t'+1))\equiv1-\m(\xi_{\widetilde{\sigma}}(t'+1))+\m(\xi_{\widetilde{\sigma}}(t+1)),$$
a contradiction. Therefore we have proved (iii).
\end{proof}

\subsection{Domain of $f$}

The aim of this section is to define the domain of the functions $F$ and $F_0$ mentioned at the beginning of section \ref{constructionoff}. As we have anticipated in the introduction, this domain is a $(r+1)$-torus $\widehat{T}=D/\!\sim$, where $D$ is the closure of a fundamental domain for $\R^{r+1}$ under the translation action of its subgroup $\Z^{r+1}$, and $\sim$ identifies points of $D$ lying in the same $\Z^{r+1}$-orbit.

Note that $\xi_\sigma(r+1)=1$, so $\m(\xi_\sigma(r+1))=0$ for all $\sigma\in S_r$. It follows that $(r+1,t)$ satisfies condition \eqref{orden1} and \eqref{orden3} for all $t\in\{1,\dots,r\}$ and $\sigma\in S_r$.

Putting $z=\xi_\sigma(t)$ in Lemma~\ref{mprop}~(iii), we have that $$\arg\!\Big(\xi_\sigma(t)\cdot\e\big(2\pi i\m(\xi_\sigma(t))/N\big)\Big) \in [0, 2\pi/N)$$ for all $t\in\{1,\dots,r+1\}$ and $\sigma \in S_r$, so there is a unique $d_{t,\sigma}\in\Z$
such that
\begin{align}
\sum_{j=1}^t \arg\!\big(\varepsilon_{\sigma(j-1)}^{(1)}\big) - \sum_{j=1}^{r+1}\arg\!\big(\varepsilon_{\sigma(j-1)}^{(1)}\big) &+
\frac{2\pi}{N}\m(\xi_\sigma(t)) + 2\pi d_{t,\sigma} \ \in \ \Big[0, \ \frac{2\pi}{N}\Big) \label{dtsigma},
\end{align}
where $\sum_{j=1}^{r+1}\arg\!\big(\varepsilon_{\sigma(j-1)}^{(1)}\big)$ is independent of $\sigma$, and the summands corresponding to $j=1$ are $0$ by definition. Note that in \eqref{rosigma}
\begin{equation}\label{fijoporrho}
\rho_\sigma(r+1)=r+1.
\end{equation}

Considering \eqref{dtsigma}, and the permutation $\rho_\sigma\in S_{r+1}$ defined in \eqref{rosigma}, we can make the following definition.

\begin{definition}\label{phi}
For \ $t\in\{1,\dots,r+1\}$, \ $\mu=(\sigma,q,n)\in \widetilde{S}_r$, \ and \ $j\in\Z$, \ we let
$$a(t,\sigma,j):=\frac{1}{N}\big(\m(\xi_\sigma(t))+j\big)+d_{t,\sigma} \ \in \ \R.$$
Also, we define \ $\phi_{t,\mu}$ \ and \ $\phi_{r+2,\mu} \ \in \ \R^{r+1}$ \ by putting
\begin{align*}
&\phi_{t,\mu}=\phi(t,\sigma,q,n):= \begin{cases}\sum\limits_{j=1}^te_{\sigma(j-1)}\ +\ a(t,\sigma,n)\cdot e_{r+1}\  & \mathrm{if}\ t\not\prec_\sigma\rho_\sigma(q)
,\\
  \sum\limits_{j=1}^te_{\sigma(j-1)}\ +\ a(t,\sigma,n+1)\cdot e_{r+1}\ & \mathrm{if}\ t\prec_\sigma\rho_\sigma(q),
\end{cases}\\
&\phi_{r+2,\mu}=\phi(r+2,\sigma,q,n):=\sum_{j=1}^{\rho_\sigma(q)}e_{\sigma(j-1)}\ +\ a(\rho_\sigma(q),\sigma,n+1)\cdot e_{r+1}.
\end{align*}
Here, \ $e_{\sigma(0)}=e_0:=0$ \ by definition, and \ $\{e_i\}_{i=1}^{r+1}$ \ is the usual basis of $\R^{r+1}$.
\end{definition}

\begin{lemma}\label{phiind}
For each \ $\mu=(\sigma,q,n)\in \widetilde{S}_r$, \ the set \ $\{\phi_{t,\mu}\}_{t=1}^{r+2}$ \ is affinely independent.
\end{lemma}

\noindent Recall that a subset $\{w_i\}_{i=1}^n$ of a real vector space $V$ is affinely independent if and only if the set
$\{w_i-w_j\}_{\substack{1\leq i\leq n\\ i\not=j}}$ is $\R$-linearly independent for some fixed $j\in\{1,\dots,n\}$.

\begin{proof}
We will prove Lemma~\ref{phiind} by showing that $\{\phi_{t,\mu}-\phi_{r+2,\mu}\}_{t=1}^{r+1}$ is linearly independent in $\R^{r+1}$. Of course, the set $\{e_{\sigma(t-1)}\}_{t=1}^{r+1}$ is affinely independent in $\R^{r+1}$. This implies that the set $\{v_{t}\}_{\substack{1\leq
t\leq r+1\\ t\not= \rho_\sigma(q)}}$ is linearly independent, where $v_t$ equals
$\sum_{j=1}^te_{\sigma(j-1)} - \sum_{j=1}^{\rho_\sigma(q)}e_{\sigma(j-1)}$,
with the product $v_{t}e_{r+1}^T=0$ (here, $e_{r+1}^T$ denote the transpose of $e_{r+1}$). Then, the determinant of the $(r+1)\times(r+1)$ matrix
\begin{equation}\label{19}
(v_{t} \ ; \ \phi_{\rho_\sigma(q),\mu}-\phi_{r+2,\mu})_{\substack{1\leq t\leq r+1\\ t\not= \rho_\sigma(q)}},
\end{equation}
whose columns are the vectors $v_{t}$ and $\phi_{\rho_\sigma(q),\mu}-\phi_{r+2,\mu}$, is zero because
\begin{equation}\label{1.16x}
\phi_{\rho_\sigma(q),\mu}-\phi_{r+2,\mu}=-\frac{1}{N}\cdot e_{r+1},
\end{equation}
by Definition~\ref{phi}. Furthermore,
\begin{align}
\nonumber
\phi_{t,\mu}-\phi_{r+2,\mu}:= \begin{cases}v_{t} + \big(a(t,\sigma,n)-a(\rho_\sigma(q),\sigma,n+1)\big)\cdot e_{r+1}\  & \mathrm{if}\ \rho_\sigma(q)\prec_\sigma t ,\\
  v_{t} + \big(a(t,\sigma,n+1)-a(\rho_\sigma(q),\sigma,n+1)\big)\cdot e_{r+1}\ & \mathrm{if}\ t\prec_\sigma\rho_\sigma(q),
\end{cases}
\end{align}
for all $1\leq t\leq r+1$ with $t\not=\rho_\sigma(q)$. Hence, we can transform the matrix in \eqref{19} into the matrix $(\phi_{t,\mu}-\phi_{r+2,\mu})_{t=1}^{r+1}$
using elementary operations. Therefore, the set $\{\phi_{t,\mu}-\phi_{r+2,\mu}\}_{t=1}^{r+1}$ is linearly independent in
$\R^{r+1}$.
\end{proof}

The above lemma implies that every non-empty subset of $\{\phi_{t,\mu}\}_{t=1}^{r+2}$ is affinely independent in $\R^{r+1}$.

Now we establish some notation. If $w_1,\dots,w_\ell$ are elements of a real vector space $W$, then the (closed) polytope they generate is the set of convex sums
$$
P=P(w_1,\dots,w_\ell):=\Big\{ w\in W \Big|\,w=\sum_{t=1}^{\ell} b_tw_t,\ \
b_t\ge0,\ \ \sum_{t=1}^{\ell}b_t=1  \Big\}
$$
($P(\varnothing):=\varnothing$). In general, if $w=\sum_{t=1}^{\ell} b_tw_t$, $b_t\in\R$, and $\sum_{t=1}^{\ell}b_t=1$, then the $b_t$ are called barycentric coordinates of $w$
with respect to the set $w_1,\dots,w_\ell$. If $w_1,\dots,w_\ell$ is affinely independent, the barycentric coordinates of $w$ are uniquely determined by $w$, so we can write $b_t=b_t(w)$.

For any $\sigma\in S_r$ and any $t\in\{1,\dots,r+1\}$, let
$$v_t:=\sum\limits_{j=1}^te_{\sigma(j-1)}+v_t^{(r+1)}e_{r+1}\,, \qquad\qquad w_t:=\sum\limits_{j=1}^te_{\sigma(j-1)}+w_t^{(r+1)}e_{r+1}.$$
If $x,y\in\R^{r+1}$ are vectors such that $x^{(j)}=y^{(j)}$ $(1\leq j\leq r)$, then we claim that
\begin{equation}\label{bari=}
x=\sum_{t=1}^{r+1} b_tv_t, \quad y=\sum_{t=1}^{r+1} b'_tw_t, \quad \sum_{t=1}^{r+1}b_t=\sum_{t=1}^{r+1}b'_t=1 \qquad \Longrightarrow \qquad b_t=b'_t
\end{equation}
for any $t\in\{1,\dots,r+1\}$. To prove \eqref{bari=}, it is easy to verify that
$$b_{r+1}+\dots+b_{j+1}=b'_{r+1}+\dots+b'_{j+1} \qquad\qquad (1\leq j\leq r)$$
by multiplying $x$ and $y$ by $e_{\sigma(j)}^T$ since $x^{(j)}=y^{(j)}$ $(1\leq j\leq r)$. Hence, $b_t=b'_t$ $(2\leq j\leq r+1)$, and also $b_1=b'_1$ since the sum of the barycentric coordinates is 1.

Using the above notation, we define the polytopes
\begin{equation}\label{P}
P_{1,\mu}=P_{1}(\sigma,q,n):=P\Big(\{\phi_{t,\mu}\}_{t=1}^{r+1}\Big), \qquad
P_{2,\mu}=P_{2}(\sigma,q,n):=P\Big(\{\phi_{t,\mu}\}_{\substack{1\leq t\leq r+2\\ t\not=\rho_\sigma(q)}}\Big)
\end{equation}
for $\mu=(\sigma,q,n)\in \widetilde{S}_r$. The next lemma will allow us to give an alternative description of these polytopes.

\begin{lemma}\label{lema8}
Let \ $\mu=(\sigma,q,n)\in \widetilde{S}_r$. \ If \ $x^{(1)},\dots,x^{(r)}\in[0,1]$ \ satisfy \ $x^{(\sigma(1))}\geq\dots\geq x^{(\sigma(r))}$, \
there exist unique \ $y_1, y_2\in\R$ \ such that \ $(x^{(1)},\dots,x^{(r)},y_1)\in P_{1,\mu}$ \ and \ $(x^{(1)},\dots,x^{(r)},y_2)\in P_{2,\mu}$. \ Furthermore, such \ $y_1$ \ and \ $y_2$ \ satisfy \ $y_1\leq y_2$.
\end{lemma}

\begin{proof}
Put $b_1:=1-x^{(\sigma(1))}$, $b_{t}:=x^{(\sigma(t-1))}-x^{(\sigma(t))}$ $(2\leq t\leq r)$, and $b_{r+1}:=x^{(\sigma(r))}$. Clearly, $b_t\geq 0$
for all $1\leq t\leq r+1$, and also $\sum_{t=1}^{r+1}b_t=1$. Thus $v:=\sum_{t=1}^{r+1}b_t\phi_{t,\mu} \in  P_{1,\mu}$, and we can check that $ve_{\sigma(j)}^T=x^{(\sigma(j))}$ for each $1\leq j\leq r$.
Putting $y_1:=ve_{r+1}^T$, uniqueness follows from \eqref{bari=}. Analogously, existence and uniqueness of $y_2$ follow from \eqref{1.16x}.

Now let us prove the last statement of Lemma~\ref{lema8}. Using \eqref{bari=}, note that the vectors $v_{y_1}:=(x^{(1)},\dots,x^{(r)},y_1)$ and $v_{y_2}:=(x^{(1)},\dots,x^{(r)},y_2)$ have equal barycentric coordinates with respect to the vertices $\{\phi_{t,\mu}\}_{t=1}^{r+1}$ and $\{\phi_{t,\mu}\}_{\substack{1\leq t\leq r+2\\ t\not=\rho_\sigma(q)}}$ respectively. Since these sets
differ only in the elements $\phi_{\rho_\sigma(q),\mu}$ and $\phi_{r+2,\mu}$, we deduce that $v_{y_2}-v_{y_1}$ equals $b(\phi_{r+2,\mu}-\phi_{\rho_\sigma(q),\mu})=\frac{b}{N}e_{r+1},$
where $b$ is a barycentric coordinate of $v_{y_2}$ with respect to $\{\phi_{t,\mu}\}_{\substack{1\leq t\leq r+2\\ t\not=\rho_\sigma(q)}}$. Therefore,
$(v_{y_2}-v_{y_1})e_{r+1}^T=y_2-y_1\geq 0$.
\end{proof}

Since both $y_1$ and $y_2$ in Lemma~\ref{lema8} depend on $\mu\in \widetilde{S}_r$ and on $x^{(1)},\dots,x^{(r)}\in[0,1]$, we shall write $y_{i,\mu}(x)=y_{i}(\sigma,q,n)(x):=y_{i}$ for any $i=1,2$, $\mu=(\sigma,q,n)\in \widetilde{S}_r$, and $x\in[0,1]^r\times\R$ such that $x^{(\sigma(1))}\geq x^{(\sigma(2))}\geq\dots\geq x^{(\sigma(r))}$.

From the proof of Lemma~\ref{lema8}, note that $v=v(x)=\sum_{t=1}^{r+1}b_t(x)\phi_{t,\mu}$ is continuous in $x$ since each of the $b_t=b_t(x)$ is continuous in $x$. Therefore, $y_1(x)=v(x) e_{r+1}^T$ is also continuous. The same holds for $y_2(x)$.

\begin{definition}\label{Dsigmaqn}
For \ $\mu=(\sigma,q,n)\in \widetilde{S}_r$, \ define
$$
D_{\mu}=D(\sigma,q,n):=\left\{x\in[0,1]^r\times\R \left|\begin{array}{c}
x^{(\sigma(1))}\geq x^{(\sigma(2))}\geq\dots\geq x^{(\sigma(r))},\\
y_{1,\mu}(x)\leq x^{(r+1)}\leq y_{2,\mu}(x).
\end{array}\right.\right\}.
$$
\end{definition}

\noindent Lemma~\ref{lema8} implies that $y_{i,\mu}$ is an affine function, \ie
$$
(1-t)y_{i,\mu}(x)+ty_{i,\mu}(z)=y_{i,\mu}\big((1-t)x + tz\big)\qquad (i=1,2; \ t\in[0,1]; \ x,z\in D_{\mu}),
$$
since $P_{1,\mu}$ and $P_{2,\mu}$ are convex sets. Hence, $D_{\mu}$ is a convex set. Furthermore, $\phi_{t,\mu}\in D_{\mu}$ $(1\leq t\leq r+2)$. In order to verify this, note that
$$
\phi_{t,\mu}^{(\sigma(j))} = \begin{cases} 1 \  & \mathrm{if}\ j<t ,\\
  0 \ & \mathrm{if} \ j\geq t,
\end{cases} \qquad \phi_{r+2,\mu}^{(\sigma(j))} = \begin{cases} 1 \ & \mathrm{if} \ j<\rho_\sigma(q),\\
  0 \ & \mathrm{if} \ j\geq \rho_\sigma(q)
\end{cases}\qquad (1\leq j,t\leq r+1),
$$
and
$$\phi_{t,\mu}^{(r+1)}=y_{1,\mu}(\phi_{t,\mu}), \qquad\qquad \phi_{r+2,\mu}^{(r+1)}=y_{2,\mu}(\phi_{r+2,\mu}) \qquad\qquad (1\leq t\leq r+1)$$
since $\phi_{t,\mu}\in P_{1,\mu}$ and $\phi_{r+2,\mu}\in P_{2,\mu}$. Thus, $P(\phi_{1,\mu},\dots,\phi_{r+2,\mu})\subset D_{\mu}$. On the other hand, each $v\in D_{\mu}$ is contained in the straight line passing through
$$\big(v^{(1)},\dots,v^{(r)},y_{1,\mu}(v)\big)\in P_{1,\mu} \qquad\qquad \mathrm{and} \qquad\qquad\big(v^{(1)},\dots,v^{(r)},y_{2,\mu}(v)\big)\in P_{2,\mu},$$
so
\begin{equation}\label{Dsigmaqn2}
P(\phi_{1,\mu},\phi_{2,\mu},\dots,\phi_{r+2,\mu}) = D_{\mu}.
\end{equation}
Therefore Lemma~\ref{phiind} implies that the $D_\mu$ are $(r+1)$-simplices.

We shall see next that we can put the $D_{\mu}=D(\sigma,q,n)$ a top one another so that every intersection of two adjacent simplices is an $r$-simplex. We shall do this by fixing $\sigma\in S_r$ and varying $q$ and $n$. More precisely, fix $\sigma\in S_r$. From \eqref{rosigma} and Definition~\ref{phi} we have
\begin{equation}\label{identidadesphi}
\phi(t,\sigma,q-1,n)=\phi(t,\sigma,q,n) \qquad \mathrm{and} \qquad \phi(\rho_{\sigma(q)},\sigma,q-1,n)=\phi(r+2,\sigma,q,n)
\end{equation}
$(1\leq t\leq r+1, \ t\not=\rho_\sigma(q); \ 2\leq q\leq r+1; \ 0\leq n\leq N-1)$. Also
$$\phi(t,\sigma,r+1,n+1)=\phi(t,\sigma,1,n) \qquad \mathrm{and} \qquad \phi(\rho_{\sigma}(1),\sigma,r+1,n+1)=\phi(r+2,\sigma,1,n)$$
$(1\leq t\leq r+1, \ t\not=\rho_{\sigma}(1); \ 0\leq n\leq N-2)$. Thus we have
\begin{align*}
&P_{1}(\sigma,q-1,n)=P_{2}(\sigma,q,n) &&(2\leq q\leq r+1; \ 0\leq n\leq N-1),\\
&P_{1}(\sigma,r+1,n+1)=P_{2}(\sigma,1,n) &&(0\leq n\leq N-2).
\end{align*}
Applying Lemma~\ref{lema8} to the last two identities, we get the following chain of inequalities for any $x\in [0,1]^r\times\R$ such that $x^{(\sigma(1))}\geq\dots\geq x^{(\sigma(r))}$.
\begin{align}
\nonumber
&y_{1}(\sigma,r+1,0)(x)\leq y_{1}(\sigma,r,0)(x)\leq\dots\leq y_{1}(\sigma,1,0)(x)\leq \\
\nonumber
&y_{1}(\sigma,r+1,1)(x)\leq y_{1}(\sigma,r,1)(x)\leq\dots\leq y_{1}(\sigma,1,1)(x)\leq \\
\nonumber
&\vdots\\
&y_{1}(\sigma,r+1,N-1)(x)\leq y_{1}(\sigma,r,N-1)(x)\leq\dots\leq y_{2}(\sigma,1,N-1)(x). \label{25}
\end{align}
Note that this chain ends with $y_{2}(\sigma,1,N-1)(x)$, the only ``link'' of the chain indexed by 2 instead 1. However, we can index $y_{2}(\sigma,1,N-1)(x)$ by 1 since
\begin{align}\nonumber
P_{2}(\sigma,1,N-1)=P\Big(\big\{\phi(t,\sigma,1,N-1)\big\}_{\substack{1\leq t\leq r+2\\ t\not=\rho_{\sigma(1)}}}\Big)
&=P\Big(\big\{e_{r+1}+\phi(t,\sigma,r+1,0)\big\}_{t=1}^{r+1}\Big)\\
&=e_{r+1}+P_{1}(\sigma,r+1,0), \label{politoposciclicos}
\end{align}
so $y_{2}(\sigma,1,N-1)(x)=1+y_{1}(\sigma,r+1,0)(x)$,
where $y_{1}(\sigma,r+1,0)(x)$ is the first link of the chain. Hence, in the following we shall write
\begin{equation}\label{ysigmaqn}
y_{\mu}(x):=y_{1,\mu}(x) \qquad (\mu\in \widetilde{S}_r; \ x\in[0,1]^{r}\times\R \ \mathrm{with} \ x^{(\sigma(1))}\geq\dots\geq x^{(\sigma(r))}).
\end{equation}

Considering the above, put
\begin{equation}\label{D}
D:=\bigcup_{\sigma\in S_r}\triangle_\sigma=\bigcup_{\mu\in \widetilde{S}_r} D_{\mu},
\end{equation}
where
$$
\triangle_\sigma:=\left\{x\in [0,1]^r\times\R \left|\begin{array}{c}
 x^{(\sigma(1))}\geq x^{(\sigma(2))}\geq\dots\geq x^{(\sigma(r))},\\
y(\sigma,r+1,0)(x)\leq x^{(r+1)}\leq 1+y(\sigma,r+1,0)(x).
\end{array}\right.\right\}.
$$
We are interested in three properties of $D$. In the first place, $D$ is a finite union of compact sets, and so is a compact subset of
$\R^{r+1}$. Secondly, $D$ is the topological closure of a fundamental domain for $\R^{r+1}$ under the translation action of its subgroup $\Z^{r+1}$. Indeed, the set
\begin{equation}\label{D2}
\mathfrak{D}:=\bigcup_{\sigma\in S_r}\mathfrak{D}_\sigma,
\end{equation}
$$
\mathfrak{D}_\sigma:=\left\{x\in [0,1)^r\times\R \left|\begin{array}{c}
x^{(\sigma(1))}\geq x^{(\sigma(2))}\geq\dots\geq x^{(\sigma(r))},\\
y(\sigma,r+1,0)(x)\leq x^{(r+1)}< 1+y(\sigma,r+1,0)(x).
\end{array}\right.\right\},
$$
is such a fundamental domain. The quotient space
\begin{equation}\label{torocanonico}
\widehat{T}:=D/\sim
\end{equation}
is homeomorphic to the standard $(r+1)$-torus $\R^{r+1}/\Z^{r+1}$, where $\sim$ is the identification of elements in the same orbit with respect to this action.

Finally, we will show that the $D_{\mu}$ form a simplicial decomposition of $D$. Since we have \eqref{D}, we only need to verify that the $D_{\mu}$ intersect each other in faces. Recall that a face of a polytope $P$ is the polytope generated by a subset of only its vertices. Now we need some technical remarks.

\begin{lemma}\label{Lema21previos}
Let \ $\mu=(\sigma,q,n),\, \mu'=(\sigma',q',n')\in \widetilde{S}_r$. \ Put
$$\mathfrak{B}_{\sigma,\sigma'}:=\{1\}\cup\Big\{2\le t\le r+1 \ \Big| \ \{\sigma(j-1) | 2\leq j\leq t\}=\{\sigma'(j-1) | 2\leq j\leq t\}\Big\}.$$
Then the following hold.
\begin{enumerate}[$(i)$]
\item If \ $w\in\R^r$ \ satisfies \ $w^{(\sigma(1))}\geq\dots\geq w^{(\sigma(r))}$ \ and \ $w^{(\sigma'(1))}\geq\dots\geq
w^{(\sigma'(r))}$, \
then \ $w^{(\sigma(j))}=w^{(\sigma'(j))}$ \ for all \ $1\leq j\leq r$.

\item Let \ $v\in[0,1]^r\times\R$ \ with \ $v^{(\sigma(1))}\geq\dots\geq v^{(\sigma(r))}$ \ and \ $v^{(\sigma'(1))}\geq\dots\geq v^{(\sigma'(r))}$. \ If
$$v=\sum_{t=1}^{r+1}b_t\phi_{t,\mu} \qquad {\text or} \qquad v=(b_{\rho_{\sigma}(1)}/N)e_{r+1}+\sum_{t=1}^{r+1}b_t\phi(t,\sigma,1,N-1)$$
for some \ $b_t\in\R$, \ then \ $b_t=0$ \ whenever \ $t\not\in\mathfrak{B}_{\sigma,\sigma'}$.

\item If \ $t, t'\in \mathfrak{B}_{\sigma,\sigma'}$, \ then \ $t\prec_\sigma t'$ \ if and only if \ $t\prec_{\sigma'} t'$.

\item If \ $t\in\mathfrak{B}_{\sigma,\sigma'}$, \ then \ $a(t,\sigma,\ell)=a(t,\sigma',\ell)$ \ for any $\ell\in\Z$.
\end{enumerate}
\end{lemma}

\begin{proof}
To prove (i), for the sake of contradiction suppose that $w^{(\sigma(j))}\not=w^{(\sigma'(j))}$ for some $j\in\{1,\dots,r\}$ (say $w^{(\sigma(j))}<w^{(\sigma'(j))}$). Since $w^{(\sigma(j))}<w^{(\sigma'(j))}\leq w^{(\sigma'(i))}$ for all $i\in\{1,\dots,j\}$, there are at least $j$ coordinates of $w$ greater than $w^{(\sigma(j))}$. But this contradicts $w^{(\sigma(1))}\ge\dots\geq w^{(\sigma(j))}$,
which implies that there are at most $j-1$ of such coordinates.

Let us prove (ii). If $t\in\{2,\dots,r+1\}$ is such that $\{\sigma(j-1) \ | \ 2\leq j\leq t\}\not=\{\sigma'(j-1) \ | \ 2\leq j\leq t\}$ \ (note that $t\not=r+1$) then there exists $j\in\{2,\dots,t\}$ such that $\sigma(j-1)=\sigma'(i)$ with $i\in\{t,\dots,r\}$. This implies that
$$1\geq\dots\geq v^{(\sigma'(j-1))}\geq\dots\geq v^{(\sigma'(t-1))}\geq\dots\geq v^{(\sigma(j-1))}\geq\dots\geq 0.$$
Since $v^{(\sigma'(j-1))}=v^{(\sigma(j-1))}$ \big(by (i)\big), we have $v^{(\sigma'(t-1))}=v^{(\sigma'(t))}$. Hence,
$v^{(\sigma(t-1))}=v^{(\sigma(t))}$, and so we conclude $b_t=0$ from the identity
$$
v^{(\sigma(\ell))}=ve_{\sigma(\ell)}^T=b_{r+1}+\dots +b_{\ell+1} \qquad\qquad (1\le\ell\le r).
$$

To show (iii), note that $\xi_\sigma(t)=\xi_{\sigma'}(t)$ and $\xi_\sigma(t')=\xi_{\sigma'}(t')$ for all $t,t'\in\mathfrak{B}_{\sigma,\sigma'}$. Thus,
assertion (iii) follows from conditions \eqref{orden1}, \eqref{orden2}, and \eqref{orden3}.

Assertion (iv) follows directly from \eqref{dtsigma} and from Definition~\ref{phi}.
\end{proof}

\begin{lemma}\label{Lema21a}
Let \ $\mu=(\sigma,q,n),\,\mu'=(\sigma',q',n')\in \widetilde{S}_r$. \ If \ $A\subset\{\phi_{t,\mu}\}_{t=1}^{r+1}$ \ and \ $A'\subset\{\phi_{t,\mu'}\}_{t=1}^{r+1}$, \ then \ $P(A)\cap P(A') = P(A\cap A')$. \
Also, this assertion remains valid if we replace both (or one of the) sets \ $\{\phi_{t,\mu}\}_{t=1}^{r+1}$ \ and \ $\{\phi_{t,\mu'}\}_{t=1}^{r+1}$ \ by \ $\{\phi(t,\sigma,1,N-1)\}_{\substack{1\leq t\leq r+2\\ t\not=\rho_\sigma(1)}}$ \ and \ $\{\phi(t,\sigma',1,N-1)\}_{\substack{1\leq t\leq r+2\\
t\not=\rho_{\sigma'}(1)}}$ \ respectively.
\end{lemma}

\begin{proof}
First note that in any case $A \cap A' \subset P(A)\cap P(A')$, and so $P(A \cap A') \subset P(A)\cap P(A')$ since $P(A)\cap P(A')$ is a convex set. Thus we have only to prove the reverse inclusion. Suppose that $P(A)\cap P(A')\not=\varnothing$ (otherwise the inclusion is obvious). Take $v\in P(A)\cap P(A')$ and then expand it in its barycentric coordinates with respect to $A$ and $A'$: $v=\sum_{t=1}^{r+1}b_t\phi_{t,\mu}=\sum_{t=1}^{r+1}b'_t\phi_{t,\mu'}$,
where $b_t, b'_t\geq 0$; $\sum_{t=1}^{r+1}b_t=\sum_{t=1}^{r+1}b'_t=1$; $b_t=0$ if $\phi_{t,\mu}\not\in A$, and $b'_t=0$ if $\phi_{t,\mu'}\not\in
A'$. From \eqref{Dsigmaqn2}, we know that $P(A)\subset D_{\mu}$ and $P(A')\subset D_{\mu'}$, so $v^{(\sigma(1))}\geq\dots\geq v^{(\sigma(r))}$ and $v^{(\sigma'(1))}\geq\dots\geq v^{(\sigma'(r))}$. Hence, using \eqref{bari=} and Lemma~\ref{Lema21previos}~(ii) we have
\begin{equation}\label{27}
\sum_{t\in\mathfrak{B}_{\sigma,\sigma'}} b_t(\phi_{t,\mu}-\phi_{t,\mu'})=0.
\end{equation}
Without loss of generality we can assume $n\leq n'$. First suppose $n<n'$; we claim that $v\in P(A\cap A')$. Indeed, Definition~\ref{phi} and Lemma~\ref{Lema21previos}~(iv) imply that
$(\phi_{t,\mu}-\phi_{t,\mu'})e_{r+1}^T\leq 0$
for all $t\in\mathfrak{B}_{\sigma,\sigma'}$. Therefore, $b_t=b'_t>0$ implies that $\phi_{t,\mu}=\phi_{t,\mu'}$
by \eqref{27}, and then $v\in P(A\cap A')$. If $n=n'$, for the sake of contradiction suppose that there exist $t,t'\in\mathfrak{B}_{\sigma,\sigma'}$ such that
\begin{equation}\label{29}
\big(\phi_{t,\mu}-\phi(t,\sigma',q',n)\big)e_{r+1}^T< 0\qquad \mathrm{and} \qquad
\big(\phi_{t',\mu}-\phi(t',\sigma',q',n)\big)e_{r+1}^T> 0.
\end{equation}
Thus, $t\not= t'$. If $t\prec_\sigma t'$, we have that $t\prec_\sigma t'\prec_\sigma \rho_\sigma(q)$ by Definition~\ref{phi}, Lemma~\ref{Lema21previos}~(iv), and the second inequality of \eqref{29}. That is,
$\big(\phi_{t,\mu}-\phi(t,\sigma',q',n)\big)e_{r+1}^T\geq 0$, which contradicts the first inequality of \eqref{29}. If $t'\prec_\sigma t$ we have $t'\prec_{\sigma'} t\prec_{\sigma'}\rho_{\sigma'}(q')$ by Lemma~\ref{Lema21previos}~(iii), Definition~\ref{phi}, Lemma~\ref{Lema21previos}~(iv), and the first inequality of \eqref{29}. Therefore, $\big(\phi_{t',\mu}-\phi(t',\sigma',q',n)\big)e_{r+1}^T\leq 0$, which contradicts the second inequality of \eqref{29}. Thus $v\in P(A\cap A')$, as in the case $n<n'$.

To prove the last part of the lemma, let us verify the case
$A\subset\{\phi(t,\sigma,1,N-1)\}_{\substack{1\leq t\leq r+2\\ t\not=\rho_\sigma(1)}}$ and $A'\subset\{\phi_{t,\mu'}\}_{t=1}^{r+1}$. Let $v\in P(A)\cap
P(A')$. Expanding $v$ in barycentric coordinates,
$$v=(b_{\rho_{\sigma}(1)}/N)e_{r+1}+\sum_{t=1}^{r+1}b_t\phi(t,\sigma,1,N-1)=\sum_{t=1}^{r+1}b'_t\phi_{t,\mu'}.$$
Using \eqref{bari=} and Lemma~\ref{Lema21previos}~(ii), we conclude that $\sum_{t\in\mathfrak{B}_{\sigma,\sigma'}} b_t\delta_t=0$,
where
$$
\delta_t := \begin{cases} \phi(t,\sigma,1,N-1)-\phi_{t,\mu'} \  & \mathrm{if} \quad t\not=\rho_{\sigma}(1) ,\\
  \phi(r+2,\sigma,1,N-1)-\phi_{\rho_{\sigma}(1),\mu'} \ & \mathrm{if} \quad t=\rho_{\sigma}(1).
\end{cases}
$$
In both cases, Definition~\ref{phi} and Lemma~\ref{Lema21previos}~(iv) show that $\delta_te_{r+1}^T\geq 0$. Therefore,
$\delta_t=0$ whenever $b_t>0$, which implies that $v\in P(A\cap A')$.

The case $A\subset\{\phi_{t,\mu}\}_{t=1}^{r+1}$ and $A'\subset\{\phi(t,\sigma',1,N-1)\}_{\substack{1\leq t\leq r+2\\ t\not=\rho_{\sigma'}(1)}}$ follows as the previous one by symmetry. Finally, the case $A\subset\{\phi(t,\sigma,1,N-1)\}_{\substack{1\leq t\leq r+2\\ t\not=\rho_{\sigma}(1)}}$ and $A'\subset\{\phi(t,\sigma',1,N-1)\}_{\substack{1\leq
t\leq r+2\\ t\not=\rho_{\sigma'}(1)}}$ follows by using \eqref{P}, \eqref{politoposciclicos}, and the first part of the lemma.
\end{proof}

\begin{lemma}\label{Lema21b}
Let \ $\mu=(\sigma,q,n),\,\mu'=(\sigma',q',n')\in \widetilde{S}_r$. \ Then there exist \ $\widehat{q}\in\{1,\dots,r+1\}$ \ and \ $\widehat{n}\in\{0,\dots,N-1\}$ \ such that \ $y(\sigma,q,n)(v)$ \ equals either \ $y(\sigma',\widehat{q},\widehat{n})(v)$ \ or \ $1+y(\sigma',r+1,0)(v)$ \
for all \ $v\in [0,1]^r\times\R$ \ satisfying \ $v^{(\sigma(1))}\geq\dots\geq v^{(\sigma(r))}$ \ and \ $v^{(\sigma'(1))}\geq\dots\geq v^{(\sigma'(r))}$. \
Moreover, \ $y(\sigma,r+1,0)(v)=y(\sigma',r+1,0)(v)$ \ for all such \ $v$.
\end{lemma}

\begin{proof}
Consider the following three cases.

\noindent \emph{Case 1}. Suppose $t\prec_\sigma\rho_\sigma(q)$ for all $t\in \mathfrak{B}_{\sigma,\sigma'}$,
and $n\leq N-2$. Using Definition~\ref{phi}, Lemma~\ref{Lema21previos}~(iv) and \eqref{fijoporrho}, we have $\phi_{t,\mu}=\phi(t,\sigma',r+1,n+1)$ for all $t\in\mathfrak{B}_{\sigma,\sigma'}$
($\widehat{q}=r+1$ and $\widehat{n}=n+1$) since $t\not\prec_{\sigma'} r+1$ for all $t\in\{1,\dots,r+1\}$.

\noindent \emph{Case 2}. Suppose
$t\prec_\sigma\rho_\sigma(q)$ for all $t\in \mathfrak{B}_{\sigma,\sigma'}$, and $n=N-1$. Since $a(t,\sigma,N)=1+a(t,\sigma,0)$, we get that $\phi_{t,\mu}=e_{r+1}+\phi(t,\sigma',r+1,0)$ for all $t\in\mathfrak{B}_{\sigma,\sigma'}$ by proceeding as in the previous case.

\noindent \emph{Case 3}. If
there exists $t\in\mathfrak{B}_{\sigma,\sigma'}$ with $t\not\prec_\sigma \rho_\sigma(q)$, put $t_0:=\min_{t\in\mathfrak{B}_{\sigma,\sigma'}}\{t\not\prec_\sigma \rho_\sigma(q)\}$. That is,
$$t\not\prec_\sigma\rho_\sigma(q), \quad t\in\mathfrak{B}_{\sigma,\sigma'}, \quad t\not=t_0 \qquad \Longrightarrow \qquad t_0\prec_\sigma t$$
(note that $\mathfrak{B}_{\sigma,\sigma'}\not=\varnothing$ since $1\in\mathfrak{B}_{\sigma,\sigma'}$). Define $\widehat{q}:=(\rho_{\sigma'})^{-1}(t_0)$.
If $t\prec_\sigma\rho_\sigma(q)$ with $t\in\mathfrak{B}_{\sigma,\sigma'}$, then $t\prec_\sigma t_0=\rho_{\sigma'}(\widehat{q})$. If
$t\not\prec_\sigma\rho_\sigma(q)$ with $t\in \mathfrak{B}_{\sigma,\sigma'}$, then $t\not\prec_\sigma t_0$ (by the definition of $t_0$). Using Lemma~\ref{Lema21previos}~(iii) and Lemma~\ref{Lema21previos}~(iv), we have $\phi_{t,\mu}=\phi(t,\sigma',\widehat{q},n)$ for all $t\in\mathfrak{B}_{\sigma,\sigma'}$ ($\widehat{n}=n$). In particular,
$\phi(t,\sigma,r+1,0)=\phi(t,\sigma',r+1,0)$ for all $t\in\mathfrak{B}_{\sigma,\sigma'}$ ($t_0=\widehat{r+1}=r+1$).

On the other hand, if $v\in\R^{r+1}$ satisfies $v^{(\sigma(1))}\geq\dots\geq v^{(\sigma(r))}$ \ and \ $v^{(\sigma'(1))}\geq\dots\geq v^{(\sigma'(r))}$, we can write $\big(v^{(1)},\dots,v^{(r)},y_{\mu}(v)\big)$ as $\sum_{t\in\mathfrak{B}_{\sigma,\sigma'}}b_t\phi_{t,\mu}$
by \eqref{P}, Lemma \ref{lema8}, and Lemma~\ref{Lema21previos}~(ii). Therefore, the lemma follows from the three cases above.
\end{proof}

\begin{remark}\label{cadenas=}
Recall \eqref{ysigmaqn} and the chain of inequalities \eqref{25},
\begin{align*}
\nonumber
&y_{1}(\sigma,r+1,0)(x)\leq y_{1}(\sigma,r,0)(x)\leq\dots\leq y_{1}(\sigma,1,0)(x)\leq \\
\nonumber
&y_{1}(\sigma,r+1,1)(x)\leq y_{1}(\sigma,r,1)(x)\leq\dots\leq y_{1}(\sigma,1,1)(x)\leq \\
\nonumber
&\vdots\\
&y_{1}(\sigma,r+1,N-1)(x)\leq y_{1}(\sigma,r,N-1)(x)\leq\dots\leq y_{2}(\sigma,1,N-1)(x).
\end{align*}
From Lemma~\ref{Lema21b}, applied to $v\in [0,1]^r\times\R$ such that $v^{(\sigma(1))}\geq\dots\geq v^{(\sigma(r))}$ and $v^{(\sigma'(1))}\geq\dots\geq v^{(\sigma'(r))}$, we have that the links of both chains for $v$ (with $\sigma$ and $\sigma'$) are the same. This fact allows us to finish the proof of the simplicial decomposition of $D$.
\end{remark}

\begin{proposition}\label{intersecciondesimplices}
Let \ $\mu=(\sigma,q,n),\, \mu'=(\sigma',q',n')\in \widetilde{S}_r$. \ Then
\begin{equation}\label{32}
D_{\mu}\cap D_{\mu'}=P\big(\{\phi_{t,\mu}\}_{t=1}^{r+2}\cap \{\phi_{t,\mu'}\}_{t=1}^{r+2}\big),
\end{equation}
and so \ $D=\cup_{\mu\in \widetilde{S}_r}D_\mu$ \ is a simplicial decomposition of $D$.
\end{proposition}

\begin{proof}
We have only to verify that the left-hand side of \eqref{32} is contained in the right-hand side, as the other inclusion is obvious by \eqref{Dsigmaqn2}. Let $v\in D_{\mu}\cap D_{\mu'}$. From Definition~\ref{Dsigmaqn}, we have the next four possibilities:
\begin{align*}
&y_{1,\mu}(v)=v^{(r+1)}<y_{2,\mu}(v), &&y_{1,\mu}(v)<v^{(r+1)}=y_{2,\mu}(v),\\
&y_{1,\mu}(v)=v^{(r+1)}=y_{2,\mu}(v), &&y_{1,\mu}(v)<v^{(r+1)}<y_{2,\mu}(v).
\end{align*}
In the first three cases, Lemma~\ref{lema8} and Remark~\ref{cadenas=} imply that $v$ lies in the intersection of two polytopes, as in Lemma~\ref{Lema21a}, and so we have the desired inclusion.

Now suppose $y_{1,\mu}(v)<v^{(r+1)}<y_{2,\mu}(v)$. In this case, $v$ lies on the straight line passing through the points $\big(v^{(1)},\dots,v^{(r)},y_{1,\mu}(v)\big)$ and $\big(v^{(1)},\dots,v^{(r)},y_{2,\mu}(v)\big)$. Then using Remark~\ref{cadenas=}, Lemma~\ref{lema8}, and Lemma~\ref{Lema21a}, we have that these two points lie in $P\big(\{\phi_{t,\mu}\}_{t=1}^{r+2}\cap \{\phi_{t,\mu'}\}_{t=1}^{r+2}\big)$. Therefore, we have the desired inclusion by convexity.
\end{proof}

\subsection{The piecewise affine map $f$}

Now we construct the piecewise affine map $f$ mentioned at the beginning of this section. In Proposition~\ref{Escher}, we shall define $f$ as a function on $D$ \big(see \eqref{D}\big) that descends to the quotient $\widehat{T}$ described in \eqref{torocanonico}.

Consider the function
\begin{align}
\nonumber
&\ell:\big(\C\times\R^r\big)\smallsetminus\{x^{(r+1)}=0\}\longrightarrow\C\times\R^{r-1},\\
\label{ell}
&\ell(x):=\left(\frac{x^{(1)}}{x^{(r+1)}},\frac{x^{(2)}}{x^{(r+1)}},\dots,\frac{x^{(r)}}{x^{(r+1)}}\right)\in\C\times\R^{r-1},
\end{align}
valid for any $x=(x^{(1)},\dots,x^{(r+1)})\in\C\times\R^r$ with non-vanishing last coordinate $x^{(r+1)}$. We define $\widetilde{V}:=\ell(V)=\langle\widetilde{\varepsilon_1},\dots,\widetilde{\varepsilon_r}\rangle$, where
$V:=\langle\varepsilon_1,\dots,\varepsilon_r\rangle$ and $\widetilde{\varepsilon_j}:=\ell(\varepsilon_j)$ $(1\leq j\leq r)$. Here the $\varepsilon_j$ are totally positive independent units of $k$, as in Theorem~\ref{Main}. It is clear that $\widetilde{V}$ acts on $\ell\big(\mi\big)=\C^*\times\R_+^{r-1}$ by component-wise multiplication.

Let $\mu=(\sigma,q,n)\in \widetilde{S}_r$. For each $t\in\Z$, choose $\alpha_t=\alpha(t)\in k\cap\left(\C^*\times\R^r_+\right)$ as in the fourth step of the 7SA. From \eqref{aspas}, we can readily verify that
\begin{equation}\label{alfaaspas}
\alpha_t^{(1)} \, \in \ \stackrel{\circ}{\A_t} \qquad (t\in\Z).
\end{equation}
Since the set $\{\phi_{t,\mu}\}_{t=1}^{r+2}$ is affinely independent by Lemma~\ref{phiind}, we can define $A_{\mu}=A(\sigma,q,n):\R^{r+1}\to \C\times\R^{r-1}$ as
the unique affine map such that
\begin{equation}\label{afin}
A_{\mu}(\phi_{t,\mu}):=\varphi_{t,\mu} \qquad\qquad (1\leq t\leq r+2),
\end{equation}
where $\varphi_{t,\mu}:=\ell(f_{t,\mu})$ and $f_{t,\mu}$ is defined by
\begin{align}\label{ftmu}
&f_{t,\mu}=f(t,\sigma,q,n):=\begin{cases}
f_{t,\sigma}\cdot\alpha(Na(t,\sigma,n)) & \mathrm{if} \ t\nprec_\sigma\rho_\sigma(q), \\
\nonumber
f_{t,\sigma}\cdot\alpha(Na(t,\sigma,n+1)) & \mathrm{if} \ t\prec_\sigma\rho_\sigma(q),
\end{cases} \qquad (1\leq t\leq r+1)\\
&f_{r+2,\mu}=f(r+2,\sigma,q,n):=f_{\rho_\sigma(q),\sigma}\cdot\alpha(Na(\rho_\sigma(q),\sigma,n+1)).
\end{align}
Except for minors changes in notation, this definition of $f_{t,\mu}$ is the one given in \eqref{ftsigmaqj} and \eqref{ftsigmaqj2}. In fact, it is easy to verify that
\begin{equation}\label{aym}
a(t,\sigma,j+1)-a(t,\sigma,j)=1/N,  \qquad Na(t,\sigma,j)\equiv \m(\xi_\sigma(t))+j \qquad (j\in\Z)
\end{equation}

At first sight, we do not know if the image of the map $A_{\mu}$ restricted to $D_{\mu}$ is contained in $\mip$ or not. This issue will be important when we define the function $f$ by using the $A_{\mu}$. The next lemma answers this question, and will prove important in working with homotopies later. For its proof we shall use the following property of affine maps. Let $W$ and $W'$ be two real vector spaces. If $w\in W$ has barycentric coordinates $b_i$ $(1\le i\le \ell)$ with respect to $w_1,\dots,w_\ell$, and $A:W\to W'$ is an affine map with $A(w_i)=p_i$ $(1\le i\le \ell)$, then the same $b_i$ are also barycentric coordinates for $A(w)$ with respect to $p_1,\dots,p_\ell$. Therefore, using definition \eqref{afin},
\begin{equation}\label{prop.afin}
v=\sum_{t=1}^{r+2}b_t\phi_{t,\mu}, \quad b_t\in\R, \quad \sum_{t=1}^{r+2}b_t=1 \qquad \Longrightarrow \qquad A_\mu(v)=\sum_{t=1}^{r+2}b_t\varphi_{t,\mu}.
\end{equation}

\begin{lemma}\label{afindentrosemiplano}
Let \ $\mu=(\sigma,q,n)\in \widetilde{S}_r$. \ Then, for any \ $t\in\{1,\dots,r+2\}$, \ we have
$$A_{\mu}(D_\mu) \ \subset \ (f_{\rho_\sigma(q),\sigma}^{(1)}\cdot\s_{Na(\rho_\sigma(q),\sigma,n)})\times \R_+^{r-1} \ \subset \ \mip.$$
\end{lemma}

\begin{proof}
We note two properties of the map $\ell$ defined in \eqref{ell}. If $x\in\C^*\times\R_+^{r}$,
then
\begin{equation}\label{prop.ell}
\arg\!\big(\ell(x)^{(1)}\big)=\arg(x^{(1)}) \qquad \mathrm{and} \qquad \ell(x)^{(j)}\in\R_+ \qquad (2\leq j\leq r).
\end{equation}
In particular, these properties are satisfied by $x=f_{t,\mu}$, for any $t\in\Z$.

To prove the lemma, first we shall study three cases for $A_\mu(\phi_{t,\mu})$ $(1\le t\le r+2)$.

\noindent \emph{Case 1}. Suppose $1\leq t\leq r+1$ and $t\prec_\sigma\rho_\sigma(q)$. From \eqref{afin}, we have
\begin{equation}\label{50a}
A_{\mu}(\phi_{t,\mu}):=\ell(f_{t,\sigma}\cdot\alpha_{Na(t,\sigma,n+1)}).
\end{equation}
Since $t\prec_\sigma\rho_\sigma(q)$, Corollary~\ref{Cor0.5} implies that $\xi_\sigma(\rho_\sigma(q),t)\cdot\stackrel{\circ}{\A}_{1+\m(\xi_\sigma(t))}\subset\s_{\m(\xi_\sigma(\rho_\sigma(q)))}$.
Multiplying this inclusion by $\tau_1(f_{\rho_\sigma(q),\sigma})\cdot\e(2\pi in/N)$, and using \eqref{aym}, we get
$$f_{t,\sigma}^{(1)}\cdot\stackrel{\circ}{\A}_{Na(t,\sigma,n+1)}\subset f_{\rho_\sigma(q),\sigma}^{(1)}\cdot\s_{Na(\rho_\sigma(q),\sigma,n)}.$$
Then, using \eqref{50a}, \eqref{alfaaspas} and \eqref{prop.ell}, the last inclusion implies that
$A_{\mu}(\phi_{t,\mu})$ lies in $(f_{\rho_\sigma(q),\sigma}^{(1)}\cdot\s_{Na(\rho_\sigma(q),\sigma,n)})\times \R_+^{r-1}$.

\noindent \emph{Case 2}. Suppose that $1\leq t\leq r+1$ and $t\not\prec_\sigma\rho_\sigma(q)$. From \eqref{afin},
\begin{equation}\label{50b}
A_{\mu}(\phi_{t,\mu}):=\ell(f_{t,\sigma}\cdot\alpha_{Na(t,\sigma,n)}).
\end{equation}
Since $t\not\prec_\sigma\rho_\sigma(q)$, we have that \eqref{1.2} \big(respectively \eqref{aspasemiplano}\big) implies
$$\xi_\sigma(\rho_\sigma(q),t)\cdot\A_{\m(\xi_\sigma(t))}\subset \s_{\m(\xi_\sigma(\rho_\sigma(q)))}$$
whenever $\rho_\sigma(q)\prec_\sigma t$ \big(respectively $t=\rho_\sigma(q)$\big). Multiplying the last inclusion by
$\tau_1(f_{\rho_\sigma(q),\sigma})\cdot\e(2\pi in/N)$, and using \eqref{aym}, we get
$$f_{t,\sigma}^{(1)}\cdot\A_{Na(t,\sigma,n)}\subset f_{\rho_\sigma(q),\sigma}^{(1)}\cdot\s_{Na(\rho_\sigma(q),\sigma,n)}.$$
Then, using \eqref{50b}, \eqref{alfaaspas} and \eqref{prop.ell}, from the last inclusion we have that
$A_{\mu}(\phi_{t,\mu})$ lies in $\big(f_{\rho_\sigma(q),\sigma}^{(1)}\cdot\s_{Na(\rho_\sigma(q),\sigma,n)}\big)\times \R^{r-1}$.

\noindent \emph{Case 3}. Finally, if $t=r+2$, we have
\begin{equation}\label{50c}
A_{\mu}(\phi_{t,\mu}):=\ell(f_{\rho_\sigma(q),\sigma}\cdot\alpha_{Na(\rho_\sigma(q),\sigma,n+1)})
\end{equation}
\big(see \eqref{afin} and \eqref{ftmu}\big). From \eqref{aspasemiplano}, note that
$\stackrel{\circ}{\A}_{1+n+\m(\xi_\sigma(\rho_\sigma(q)))}\subset \s_{{n+\m(\xi_\sigma(\rho_\sigma(q)))}}$.
Multiplying this inclusion by $f_{\rho_\sigma(q),\sigma}^{(1)}$, and using \eqref{aym}, we have
$$f_{\rho_\sigma(q),\sigma}^{(1)}\cdot\stackrel{\circ}{\A}_{Na(\rho_\sigma(q),\sigma,n+1)}\subset
f_{\rho_\sigma(q),\sigma}^{(1)}\cdot\s_{Na(\rho_\sigma(q),\sigma,n)}.$$
Then, from \eqref{50c} and \eqref{prop.ell}, we get that
$A_{\mu}(\phi_{t,\mu})$ lies in $\big(f_{\rho_\sigma(q),\sigma}^{(1)}\cdot\s_{Na(\rho_\sigma(q),\sigma,n)}\big)\times \R^{r-1}$.

The lemma follows from \eqref{Dsigmaqn2}, \eqref{prop.afin}, and the three previous cases by the convexity of
$(f_{\rho_\sigma(q),\sigma}^{(1)}\cdot\s_{Na(\rho_\sigma(q),\sigma,n)})\times \R_+^{r-1},$
a product of convex sets.
\end{proof}

Given $\sigma\in S_r$, define $\widetilde{\sigma}\in S_r$ by $\widetilde{\sigma}(1):=\sigma(r)$,
and $\widetilde{\sigma}(j):=\sigma(j-1)$ for each $j\in\{2,\dots,r\}$. Recall the set of integers
$$B_\sigma:=\left\{1\leq t\leq r+1 \ \left| \ \m\big(\varepsilon_{\sigma(r)}^{(1)}\cdot\xi_\sigma(t)\big) \equiv
\m\big(\varepsilon_{\sigma(r)}^{(1)}\big)+\m(\xi_\sigma(t)) \right.\right\}$$
defined in Lemma~\ref{lema42}. From \eqref{Bsigmac}, for each $t\in\{1,\dots,r\}$, there exists $\kappa_{t,\sigma}'\in\Z$ such that
\begin{equation}\label{ktsigma'}
\m\big(\xi_\sigma(t)\varepsilon_{\sigma(r)}^{(1)}\big)=\begin{cases} \m(\xi_\sigma(t))+\m\big(\varepsilon_{\sigma(r)}^{(1)}\big)+\kappa_{t,\sigma}'N \
&\mathrm{if} \ t\in B_\sigma\\
\m(\xi_\sigma(t))+\m\big(\varepsilon_{\sigma(r)}^{(1)}\big)+\kappa_{t,\sigma}'N-1 \ &\mathrm{if} \ t\in B_\sigma^c\end{cases}.
\end{equation}
On the other hand, \eqref{dtsigma} and the definition of $\widetilde{\sigma}$ imply that
$$
2\pi d_{t+1,\widetilde{\sigma}}-2\pi d_{t,\sigma}+\arg\!\big(\varepsilon_{\sigma(r)}^{(1)}\big)
+\frac{2\pi}{N}\m(\xi_{\widetilde{\sigma}}(t+1))-\frac{2\pi}{N}\m(\xi_\sigma(t)) \ \in \ \left(-\frac{2\pi}{N} \ , \ \frac{2\pi}{N}\right)
$$
for any $t\in\{1,\dots,r\}$. Multiplying the last expression by $-N/2\pi$, we conclude that
$$
-Nd_{t+1,\widetilde{\sigma}}+Nd_{t,\sigma}-\frac{N}{2\pi}\arg\!\big(\varepsilon_{\sigma(r)}^{(1)}\big)-\m(\xi_{\widetilde{\sigma}}(t+1))
+\m(\xi_\sigma(t)) \ \in \ \left(-1 \ , \ 1\right).
$$
Using \eqref{m}, the ceiling function $\lceil \ \rceil$, and dividing by $N$, we get
$$d_{t+1,\widetilde{\sigma}}-d_{t,\sigma}-\frac{1}{N}\m\big(\varepsilon_{\sigma(r)}^{(1)}\big)+\frac{1}{N}\m(\xi_{\widetilde{\sigma}}(t+1))
-\frac{1}{N}\m(\xi_\sigma(t)) \ \in \ \left\{0,-\frac{1}{N}\right\}.$$
Since $N\geq 3$, Lemma~\ref{lema42}~(i) and \eqref{ktsigma'} imply that
\begin{equation}\label{revelacion}
d_{t+1,\widetilde{\sigma}}-d_{t,\sigma}+\kappa_{t,\sigma}'=0 \qquad\qquad (1\leq t\leq r, \ \sigma\in S_r).
\end{equation}

\begin{lemma}\label{antesdeEscher}
Let \ $\mu=(\sigma,q,n)\in \widetilde{S}_r$ \ be such that \ $\rho_\sigma(q)\not=r+1$, \ and define \ $\widetilde{\sigma}\in S_r$ \ by \ $\widetilde{\sigma}(1):=\sigma(r)$, \ and \
$\widetilde{\sigma}(j):=\sigma(j-1)$ \ for each \ $j\in\{2,\dots,r\}$. \ Then there exist \ $\kappa_{\mu}\in \Z$, \ $\widetilde{q}\in\{1,\dots,r+1\}$, \ and \
$\widetilde{n}\in\{0,\dots,N-1\}$, \ such that we have
\begin{align}\label{caso1escher}
&\phi_{t,\mu}+e_{\sigma(r)}+\kappa_{\mu}e_{r+1}=\phi_{t+1,\widetilde{\mu}}\in D_{\widetilde{\mu}} \qquad\qquad (1\le t\le r),\\
\label{caso2escher}
&\phi_{r+2,\mu}+e_{\sigma(r)}+\kappa_{\mu}e_{r+1}=\phi_{r+2,\widetilde{\mu}}\in D_{\widetilde{\mu}},
\end{align}
with $\widetilde{\mu}:=(\widetilde{\sigma},\widetilde{q},\widetilde{n})\in \widetilde{S}_r$.
\end{lemma}

\begin{proof}
Let $\widetilde{q}:=(\rho_{\widetilde{\sigma}})^{-1}\big(1+\rho_\sigma(q)\big)$. We will divide the proof into two cases according to whether $\rho_\sigma(q)\in B_\sigma$ or not.

\noindent\textbf{Case 1.} Suppose $\rho_\sigma(q)\in B_\sigma$. Choose $\kappa_{\mu}\in\Z$ so that
$n-\m\big(\varepsilon_{\sigma(r)}^{(1)}\big)+\kappa_{\mu} N$ lies in $\{0,\dots,N-1\}$.
Let $\widetilde{n}:=n-\m\big(\varepsilon_{\sigma(r)}^{(1)}\big)+\kappa_{\mu} N$,
and $\widetilde{\mu}:=(\widetilde{\sigma},\widetilde{q},\widetilde{n})\in \widetilde{S}_r$.
From Definition~\ref{phi}, \eqref{ktsigma'} and \eqref{revelacion}, we have for each $t\in\{1,\dots,r\}$ that
\begin{equation}\label{58}
a(t+1,\widetilde{\sigma},\widetilde{n})-a(t,\sigma,n)
=\begin{cases} \kappa_{\mu} \ &\mathrm{if} \ t\in B_\sigma,\\
\kappa_{\mu}-1/N \ &\mathrm{if} \ t\in B_\sigma^c.\end{cases}
\end{equation}
To prove \eqref{caso1escher} in this case, fix $t\in\{1,\dots,r\}$.

\noindent (i) If $t\in B_\sigma$ and $t\prec_\sigma\rho_\sigma(q)$, then Definition~\ref{phi}, \eqref{58}, Lemma~\ref{lema42}~(ii), and
\eqref{Dsigmaqn2} imply
\begin{equation}\label{59}
\phi_{t,\mu}+e_{\sigma(r)}+\kappa_{\mu}e_{r+1}=\sum_{i=1}^{t+1}e_{\widetilde{\sigma}(i-1)}+a(t+1,\widetilde{\sigma},\widetilde{n}+1)e_{r+1}
=\phi_{t+1,\widetilde{\mu}} \ \in \ D_{\widetilde{\mu}}
\end{equation}
(recall \ $a(t,\sigma,n+1)=a(t,\sigma,n)+1/N$ \ and \  $a(t+1,\widetilde{\sigma},\widetilde{n}+1)=a(t+1,\widetilde{\sigma},\widetilde{n})+1/N$).

\noindent (ii) If $t\in B_\sigma$ and $t\not\prec_\sigma\rho_\sigma(q)$, then Definition~\ref{phi}, \eqref{58}, Lemma~\ref{lema42}~(ii), and \eqref{Dsigmaqn2} imply
\begin{equation}\label{60}
\phi_{t,\mu}+e_{\sigma(r)}+\kappa_{\mu}e_{r+1}
=\sum_{i=1}^{t+1}e_{\widetilde{\sigma}(i-1)}+a(t+1,\widetilde{\sigma},\widetilde{n})e_{r+1}
=\phi_{t+1,\widetilde{\mu}} \ \in \ D_{\widetilde{\mu}}.
\end{equation}

\noindent (iii) If $t\in B_\sigma^c$ and $t\not\prec_\sigma\rho_\sigma(q)$, then Definition~\ref{phi}, \eqref{58}, Lemma~\ref{lema42}~(iii), and
\eqref{Dsigmaqn2} imply
\begin{equation}\label{61}
\phi_{t,\mu}+e_{\sigma(r)}+\kappa_{\mu}e_{r+1}
=\sum_{i=1}^{t+1}e_{\widetilde{\sigma}(i-1)}+a(t+1,\widetilde{\sigma},\widetilde{n}+1)e_{r+1}
=\phi_{t+1,\widetilde{\mu}} \ \in \ D_{\widetilde{\mu}}.
\end{equation}
Note that Lemma~\ref{lema42}~(iii) implies that $t$ can only satisfy one of the above three assumptions. Hence \eqref{caso1escher} follows from \eqref{59}, \eqref{60} and \eqref{61}.

Let us prove \eqref{caso2escher}. Using Definition~\ref{phi}, \eqref{60}, the definition of $\widetilde{q}$, and \eqref{Dsigmaqn2},
\begin{equation*}
\phi_{r+2,\mu}+e_{\sigma(r)}+\kappa_\mu e_{r+1}
=\phi_{\rho_\sigma(q)+1,\widetilde{\mu}}+(1/N)e_{r+1}
=\phi_{r+2,\widetilde{\mu}} \ \in \ D_{\widetilde{\mu}}.
\end{equation*}
Thus, we have proved Lemma~\ref{antesdeEscher} when $\rho_\sigma(q)\in B_\sigma$.

\noindent\textbf{Case 2.} Suppose $\rho_\sigma(q)\in B_\sigma^c$. Choose $\kappa_{\mu}\in\Z$
so that $n-\m\big(\varepsilon_{\sigma(r)}^{(1)}\big)+1+\kappa_{\mu} N$ lies in $\{0,\dots,N-1\}$.
Note that the definition of $\kappa_{\mu}$ here differs from the one given in case 1.
Let $\widetilde{n}:=n+1-\m\big(\varepsilon_{\sigma(r)}^{(1)}\big)+\kappa_{\mu} N$,
and $\widetilde{\mu}:=(\widetilde{\sigma},\widetilde{q},\widetilde{n})$. Then, Definition~\ref{phi}, \eqref{ktsigma'} and \eqref{revelacion}
imply for each $t\in\{1,\dots,r\}$ that
\begin{equation}\label{alturacaso2}
a(t+1,\widetilde{\sigma},\widetilde{n})-a(t,\sigma,n)=\begin{cases} \kappa_{\mu}+1/N \ &\mathrm{if} \ t\in B_\sigma\\
\kappa_{\mu}  \ &\mathrm{if} \ t\in B_\sigma^c\end{cases}.
\end{equation}
To prove \eqref{caso1escher}, fix $t\in\{1,\dots,r\}$.

\noindent (i) If $t\in B_\sigma$ and $t\prec_\sigma\rho_\sigma(q)$, then Definition~\ref{phi}, \eqref{alturacaso2}, Lemma~\ref{lema42}~(iii), and \eqref{Dsigmaqn2} imply
\begin{equation}\label{59caso2}
\phi_{t,\mu}+e_{\sigma(r)}+\kappa_{\mu}e_{r+1}
=\sum_{i=1}^{t+1}e_{\widetilde{\sigma}(i-1)}+a(t+1,\widetilde{\sigma},\widetilde{n})e_{r+1}
=\phi_{t+1,\widetilde{\mu}} \ \in \ D_{\widetilde{\mu}}.
\end{equation}

\noindent (ii) If $t\in B_\sigma^c$ and $t\prec_\sigma\rho_\sigma(q)$, then Definition~\ref{phi}, \eqref{alturacaso2}, Lemma~\ref{lema42}~(ii), and \eqref{Dsigmaqn2} imply
\begin{equation}\label{60caso2}
\phi_{t,\mu}+e_{\sigma(r)}+\kappa_{\mu}e_{r+1}
=\sum_{i=1}^{t+1}e_{\widetilde{\sigma}(i-1)}+a(t+1,\widetilde{\sigma},\widetilde{n}+1)e_{r+1}
=\phi_{t+1,\widetilde{\mu}} \ \in \ D_{\widetilde{\mu}}.
\end{equation}

\noindent (iii) If $t\in B_\sigma^c$ and $t\not\prec_\sigma\rho_\sigma(q)$, then Definition~\ref{phi}, \eqref{alturacaso2}, Lemma~\ref{lema42}~(ii), and \eqref{Dsigmaqn2} imply
\begin{equation}\label{61caso2}
\phi_{t,\mu}+e_{\sigma(r)}+\kappa_{\mu}e_{r+1}
=\sum_{i=1}^{t+1}e_{\widetilde{\sigma}(i-1)}+a(t+1,\widetilde{\sigma},\widetilde{n})e_{r+1}
=\phi_{t+1,\widetilde{\mu}} \ \in \ D_{\widetilde{\mu}}.
\end{equation}
Again, Lemma~\ref{lema42}~(iii) implies that $t$ can only satisfy one of the above three assumptions. Hence, \eqref{caso1escher} follows from \eqref{59caso2}, \eqref{60caso2} and \eqref{61caso2}.

Finally, using Definition~\ref{phi}, \eqref{60caso2}, the definition of $\widetilde{q}$, and \eqref{Dsigmaqn2},
\begin{equation*}
\phi_{r+2,\mu}+e_{\sigma(r)}+\kappa_{\mu}e_{r+1}
=\phi_{\rho_\sigma(q)+1,\widetilde{\mu}}+(1/N)e_{r+1}
=\phi_{r+2,\widetilde{\mu}} \ \in \ D_{\widetilde{\mu}},
\end{equation*}
which finishes the proof of Lemma~\ref{antesdeEscher}.
\end{proof}

We now construct our piecewise affine map $f$ with domain $D:=\cup_{\mu\in \widetilde{S}_r}D_\mu$.

\begin{proposition}\label{Escher}
There exists a continuous map \ $f:D\to \mip$ \ with the following properties:
\begin{enumerate}[$(i)$]
\item If \ $x\in D_{\mu}$, \ then \ $f(x)=A_{\mu}(x)$, \ where \ $A_\mu$ \ was defined in \eqref{afin}.

\item If \ $x\in D$ \ and \ $x+e_{r+1}\in D$, \ then \ $f(x+e_{r+1})=f(x)$.

\item If \ $x\in D$ \ and \ $x+e_j+\beta e_{r+1}\in D$ \ for some standard basis vector \ $e_j$ \ of \ $\R^{r+1}$ \ distinct from \ $e_{r+1}$, \ and \
    some \ $\beta\in\Z$, \ then
    $$f(x+e_j+\beta e_{r+1})=\widetilde{\varepsilon}_j\cdot f(x) \qquad\qquad (1\leq j\leq r).$$
\end{enumerate}
\end{proposition}

\begin{proof}
To prove the existence of $f$ and (i), we only need to show that if $x\in D_{\mu}\cap D_{\mu'}$, then $A_{\mu}(x)=A_{\mu'}(x)$. Suppose $v=\phi_{t,\mu}$ ($1\leq t\leq r+2$) is a vertex of $D_{\mu}$. Then, using Definition~\ref{phi}, \eqref{afin}, \eqref{ftmu}, and \eqref{ftsigma}, we have
\begin{equation}\label{verticeind}
A_{\mu}(v)=\ell\big(\alpha(Nv^{(r+1)})\big)\cdot\prod_{j=1}^r\widetilde{\varepsilon}_j^{v^{(j)}}.
\end{equation}
Since the last expression is independent of $\mu$, we have $A_{\mu}(v)=A_{\mu'}(v)$ whenever $v$ is a vertex of $D_{\mu}$ and of $D_{\mu'}$. But Proposition~\ref{intersecciondesimplices} implies that $D_{\mu'}\cap  D_{\mu}$ is a $d$-simplex (for some $1\leq d\leq r$) whose $d+1$ vertices are also vertices of $ D_{\mu}$ and of $D_{\mu'}$. An affine map on a $d$-simplex is uniquely determined by its values on the $d+1$ vertices, so $A_{\mu}(x)=A_{\mu'}(x)$ for all $x\in D_{\mu'}\cap  D_{\mu}$.

To prove (ii), note that \eqref{D} shows us that $x\in P_{1}(\sigma,r+1,0)\subset D(\sigma,r+1,0)$ for some $\sigma\in S_r$. If we write $x$ in its barycentric coordinates with respect to the vertices of $P_{1}(\sigma,r+1,0)$, $x=\sum_{t=1}^{r+1}b_t\phi(t,\sigma,r+1,0)$, $b_t\geq 0$, $\sum_{t=1}^{r+1}b_t=1$,
then using Definition~\ref{phi},
\begin{align*}
e_{r+1}+x &= \sum_{t=1}^{r+1}b_t\big(e_{r+1}+\phi(t,\sigma,r+1,0)\big)\\
 &= b_{\rho_\sigma(1)}\phi(r+2,\sigma,1,N-1)+\sum_{\substack{1\leq t\leq r+1\\ t\not=\rho_\sigma(1)}}b_t\phi(t,\sigma,1,N-1) \ \in \ D(\sigma,1,N-1),
\end{align*}
where the last equality follows from \eqref{politoposciclicos}. Hence, using (i), \eqref{prop.afin}, \eqref{afin}, and \eqref{aym},
\begin{align*}
f(x) &= A(\sigma,r+1,0)(x) = \sum_{t=1}^{r+1}b_t\varphi(t,\sigma,r+1,0)\\
&= b_{\rho_\sigma(1)}\varphi(\rho_\sigma(1),\sigma,r+1,0)+\sum_{\substack{1\leq t\leq r+1\\t\not=\rho_\sigma(1)}}b_t\varphi(t,\sigma,r+1,0)\\
&= b_{\rho_\sigma(1)}\varphi(r+2,\sigma,1,N-1)+\sum_{\substack{1\leq t\leq r+1\\t\not=\rho_\sigma(1)}}b_t\varphi(t,\sigma,1,N-1)\\
&= A(\sigma,1,N-1)(x+e_{r+1})=f(x+e_{r+1}).
\end{align*}

Now let us prove (iii). Since $x\in D$ \ and \ $x+e_j+\beta e_{r+1}\in D$, we have that $x\in D_{\mu}$ for some $\mu=(\sigma,q,n)\in \widetilde{S}_r$, where we can suppose $\sigma(r)=j$ because $x^{(j)}=0$ (see Definition~\ref{Dsigmaqn}). Writing $x$ in barycentric coordinates with respect to $D_{\mu}$: $x=\sum_{t=1}^{r+2}b_t\phi_{t,\mu}$, $b_t\geq 0$, $\sum_{t=1}^{r+2}b_t=1$,
note that
$$0=x^{(j)}=x^{(\sigma(r))}=xe_{\sigma(r)}^T=\begin{cases} b_{r+2}+b_{r+1} \ &\mathrm{if} \ \rho_\sigma(q)=r+1\\
b_{r+1} \ &\mathrm{if} \ \rho_\sigma(q)\not=r+1\end{cases}.
$$
We will divide the proof into two cases, depending on whether $\rho_\sigma(q)=r+1$ or not. First suppose that $\rho_\sigma(q)\not=r+1$.
Since $\sum_{t=1}^{r+2}b_t=1$ and $b_{r+1}=0$, Lemma~\ref{antesdeEscher} implies that there exist $\kappa_{\mu}\in \Z$ and $\widetilde{\mu}=(\widetilde{\sigma},\widetilde{q},\widetilde{n})\in \widetilde{S}_r$ such that
\begin{align*}
x+e_j+\kappa_{\mu}e_{r+1} &= b_{r+2}(\phi_{r+2,\mu}+e_j+\kappa_{\mu}e_{r+1}) +\sum_{t=1}^rb_t(\phi_{t,\mu}+e_j+\kappa_{\mu}e_{r+1}) \\
&=b_{r+2}\phi_{r+2,\widetilde{\mu}} +\sum_{t=1}^rb_t\phi_{t+1,\widetilde{\mu}} \ \in \ D_{\widetilde{\mu}}\subset D,
\end{align*}
where $\widetilde{\sigma}(1):=\sigma(r)$ and $\widetilde{\sigma}(j):=\sigma(j-1)$ ($2\leq j\leq r$). Moreover, since $x+e_j+\beta e_{r+1}\in D$,
Definition~\eqref{D} implies that $\kappa_\mu\in\{\beta-1,\,\beta,\,\beta+1\}$. Therefore, using (ii),
\begin{equation}\label{betaykappa}
f(x+e_j+\beta e_{r+1})=f(x+e_j+\kappa_{\mu}e_{r+1}).
\end{equation}
Then, (i), \eqref{prop.afin}, and \eqref{afin} show that
\begin{equation}\label{Escher2}
f(x+e_j+\kappa_{\mu}e_{r+1})=A_{\widetilde{\mu}}(x+e_j+\kappa_{\mu}e_{r+1})= b_{r+2}\varphi_{r+2,\widetilde{\mu}}+\sum_{t=1}^rb_t\varphi_{t+1,\widetilde{\mu}}.
\end{equation}
On the other hand, putting $v=\phi_{r+2,\mu}$, we can use \eqref{verticeind} to compute
\begin{align}
\nonumber
&\varphi_{r+2,\widetilde{\mu}}=:A_{\widetilde{\mu}}(\phi_{r+2,\widetilde{\mu}})=A_{\widetilde{\mu}}(\phi_{r+2,\mu}+e_j+\kappa_\mu e_{r+1})\\
\nonumber
&=\ell\big(\alpha(Nv^{(r+1)}+N\kappa_\mu)\big)\cdot\prod_{1\leq i\leq r}\widetilde{\varepsilon}_i^{(v^{(i)}+e_j^{(i)})}=\widetilde{\varepsilon}_j\cdot\ell\big(\alpha(Nv^{(r+1)})\big)\cdot\prod_{1\leq
i\leq r}\widetilde{\varepsilon}_i^{v^{(i)}}\\
&=\widetilde{\varepsilon}_j\cdot\varphi_{r+2,\mu}=\widetilde{\varepsilon}_j\cdot A_\mu({\phi_{r+2,\mu}}).\label{Escher3}
\end{align}
Analogously, putting $w=\phi_{t,\mu}$ ($1\leq t\leq r$),
\begin{align}
\nonumber
&\varphi_{t+1,\widetilde{\mu}}=:A_{\widetilde{\mu}}(\phi_{t+1,\widetilde{\mu}})=A_{\widetilde{\mu}}(\phi_{t,\mu}+e_j+\kappa_\mu e_{r+1})\\
\nonumber
&=\ell\big(\alpha(Nw^{(r+1)}+N\kappa_\mu)\big)\cdot\prod_{1\leq i\leq r}\widetilde{\varepsilon}_i^{(w^{(i)}+e_j^{(i)})}=\widetilde{\varepsilon}_j\cdot\ell\big(\alpha(Nw^{(r+1)})\big)\cdot\prod_{1\leq
i\leq r}\widetilde{\varepsilon}_i^{w^{(i)}}\\
&=\widetilde{\varepsilon}_j\cdot\varphi_{t,\mu}=\widetilde{\varepsilon}_j\cdot A_\mu({\phi_{t,\mu}}) \qquad (1\leq t\leq r).\label{Escher4}
\end{align}
Thus, \eqref{Escher3}, \eqref{Escher4}, \eqref{prop.afin}, and (i) imply that the right-hand side of \eqref{Escher2} equals
\begin{align*}
\widetilde{\varepsilon}_j\cdot \Big(b_{r+2}A_\mu({\phi_{r+2,\mu}})+\sum_{t=1}^rb_tA_\mu({\phi_{t,\mu}})\Big)=\widetilde{\varepsilon}_j\cdot
A_\mu(x)=\widetilde{\varepsilon}_j\cdot f(x).
\end{align*}
This, together with \eqref{betaykappa} proves Proposition~\ref{Escher} in this case.

Now suppose that $\rho_\sigma(q)=r+1$. Then, $q=r+1$ since $\rho_\sigma(r+~1)=r+1$ \big(see \eqref{fijoporrho}\big). Note that $b_{r+2}=b_{r+1}=0$ since $b_{r+2}+b_{r+1}=0$ and $b_{r+2}, b_{r+1}\geq 0$. Here, if $x\in D_{\mu}=D(\sigma,r+1,n)$, then $x\in D(\sigma,r,n)$
since we have $\phi(t,\sigma,r+1,n)=\phi(t,\sigma,r,n)$ ($1\leq t\leq r$) by \eqref{identidadesphi}. Therefore, the proof reduces to the case
$\rho_\sigma(q)\not=r+1$.
\end{proof}

\section{Proof of Theorem~\ref{Main}}

\subsection{Maps descending to tori}

Now we define the maps $F$ and $F_0$ mentioned at the beginning of section \ref{constructionoff}. Recall the $(r+1)$-torus $\widehat{T}$, defined in \eqref{torocanonico} by identifying the elements of
$$D:=\bigcup_{\sigma\in S_r}\left\{x\in \R^{r+1} \left|\begin{array}{c}
1\geq x^{(\sigma(1))}\geq x^{(\sigma(2))}\geq\dots\geq x^{(\sigma(r))}\geq 0,\\
y(\sigma,r+1,0)(x)\leq x^{(r+1)}\leq 1+y(\sigma,r+1,0)(x).
\end{array}\right.\right\}$$
that lie in the same orbit with respect to the action on $\R^{r+1}$ of the subgroup $\Z^{r+1}$. Also recall that the set
$$\mathfrak{D}:=\bigcup_{\sigma\in S_r}\left\{x\in \R^{r+1} \left|\begin{array}{c}
1> x^{(\sigma(1))}\geq x^{(\sigma(2))}\geq\dots\geq x^{(\sigma(r))}\geq 0,\\
y(\sigma,r+1,0)(x)\leq x^{(r+1)}< 1+y(\sigma,r+1,0)(x).
\end{array}\right.\right\}$$
defined in \eqref{D2} is a fundamental domain for this action. Proposition~\ref{Escher} means that the piecewise affine map $f:D\to\mip$ descends to a continuous map $F$ between $\widehat{T}$ and the quotient $\big(\mip\big)/\widetilde{V}$ coming from the action of $\widetilde{V}:=\langle \widetilde{\varepsilon}_1,\dots,\widetilde{\varepsilon}_r\rangle$ on $\mip$ by component-wise multiplication. More precisely, $F:\widehat{T}\to \big(\mip\big)/\widetilde{V}$ is defined by the commutative diagram
\begin{equation}\label{diagramaF}
\begin{CD}
D @>f>> \mip\\
@VV\widehat{\pi}V @VV\pi V\\
\widehat{T} @>F>{\phantom{\simeq}}> \big(\mip\big)/\widetilde{V}
\end{CD},
\end{equation}
where $\widehat{\pi}$ and $\pi$ are the natural quotient maps, and $f$ was defined in Proposition~\ref{Escher}.

There is another function between $\widehat{T}$ and $\big(\mip\big)/\widetilde{V}$ that naturally comes from a function on $D$. We define $f_0:D\to \mip$ by
\begin{equation}\label{f0}
\big(f_0(x)\big)^{(j)}:=\begin{cases} \big(\widetilde{\varepsilon}_1^{(1)}\big)^{x^{(1)}}\dots \big(\widetilde{\varepsilon}_r^{(1)}\big)^{x^{(r)}}\cdot\e(2\pi i
x^{(r+1)}) \ &\mathrm{if} \ j=1,\\
\big(\widetilde{\varepsilon}_1^{(j)}\big)^{x^{(1)}}\dots \big(\widetilde{\varepsilon}_r^{(j)}\big)^{x^{(r)}} \ &\mathrm{if} \ 2\leq j\leq r,\end{cases}
\end{equation}
for all $x=(x^{(1)},\dots,x^{(r+1)})\in D$, where powers are defined using the branch of the argument in $[-\pi,\pi)$. From \eqref{f0}, it is clear that $f_0$ is a continuous map that satisfies properties (ii) and (iii) of Proposition~\ref{Escher}. Thus, $f_0$ descends to a continuous map $F_0:\widehat{T}\to \big(\mip\big)/\widetilde{V}$ defined by the commutative diagram
\begin{equation}\label{diagramaF0}
\begin{CD}
D @>f_0 >> \mip\\
@VV\widehat{\pi}V @VV\pi V\\
\widehat{T} @>F_0>{\phantom{\simeq}}> \big(\mip\big)/\widetilde{V}
\end{CD},
\end{equation}
where $\widehat{\pi}$ and $\pi$ are again the natural quotient maps, and $f_0$ was defined in \eqref{f0}.

Let us write, using upper and lower case to distinguish the slightly different domains,
\begin{align}\label{LOG}
&\LOG:\mip\to \R^r,\quad\big(\LOG(x)\big)^{(j)}:=\log|x^{(j)}| \qquad (1\leq j\leq r),\\
\label{Log}
&\Log:\mi\to \R^r,\quad\big(\Log(x)\big)^{(j)}:=\log|x^{(j)}| \qquad (1\leq j\leq r).
\end{align}
Note that the function in \eqref{Log} is used in the sixth step of the 7SA.

\begin{lemma}\label{reguladorconsigno} Let \ $\LOG$ \ and \ $\Log$ \ be the functions defined respectively in \eqref{LOG} and \eqref{Log}. Then
\begin{equation}\label{lema48}
\det\!\big(\LOG(\widetilde{\varepsilon}_1),\dots,\LOG(\widetilde{\varepsilon}_r)\big)=(r+2)\det\!\big(\Log(\varepsilon_1),\dots,\Log(\varepsilon_r)\big).
\end{equation}
In particular, none of the determinants in \eqref{lema48} vanish, both have the same sign, and \ $\Lambda:=\sum_{j=1}^r\Z\cdot\LOG(\widetilde{\varepsilon}_j)\subset\R^r$ \ is a full lattice.
\end{lemma}

\begin{proof}
We follow \cite[Lemma 19]{DF1}. Using the identity $|\varepsilon_j^{(1)}|^2\cdot\prod_{i=2}^{r}\varepsilon_j^{(i)}=1/\varepsilon_j^{(r+1)}$ $(1\leq j\leq r)$,
\eqref{lema48} reduces to showing $r+2=\det(I_r+B_r)$, where the $r\times r$ matrices $I_r$ and $B_r$ are, respectively, the identity and the matrix whose first column has only entries equal to 2 and all the other entries are 1's. But $\det(\lambda I_r-B_r)=\lambda^{r-1}\big(\lambda-(r+1)\big)$, using the eigenvalues 0 and $r+1$ of $B_r$. Setting $\lambda=-1$ concludes the proof.
\end{proof}

We will soon show that $f_0(\mathfrak{D})$ is a fundamental domain for the action of $\widetilde{V}$ on $\mip$, but we first make some geometric remarks. If $\ell(\varepsilon)=\widetilde{\varepsilon}\in \widetilde{V}$ satisfies $|\widetilde{\varepsilon}^{(j)}|=1$ for all $1\leq j\leq r$ (recall $r>0$), then
$|\varepsilon^{(1)}|=\varepsilon^{(2)}=\dots=\varepsilon^{(r+1)}$. Since $1=|N_{k/\Q}(\varepsilon)|$, we have then also $|\varepsilon^{(r+1)}|=1$. Since $\varepsilon^{(r+1)}>0$, we see that
\begin{equation}\label{unidadnormauno}
\varepsilon\in V, \quad \ell(\varepsilon)=\widetilde{\varepsilon}, \quad |\widetilde{\varepsilon}^{(j)}|=1 \quad (1\leq j\leq r) \qquad \Longrightarrow \qquad
\varepsilon=1.
\end{equation}

If $\nu:\mip\to \R_+^r$ is defined by the formula $\big(\nu(x)\big)^{(j)}:=|x^{(j)}|$ $(1\leq j\leq r)$, then Lemma~\ref{reguladorconsigno} implies that $(\nu\circ f_0)(\mathfrak{D})$ is a fundamental domain for the action of $\nu(\widetilde{V})$ on $\R_+^r$. Since the exponential $\e(2\pi i x^{(r+1)})$ in
\eqref{f0}, restricted to $\mathfrak{D}$, runs over the unit circle exactly once, it is clear that each $x\in\mip$ is in the orbit of some
$\widetilde{x}\in f_0(\mathfrak{D})$ under the action of $\widetilde{V}$. Furthermore, if there are two such elements $\widetilde{x}, \widetilde{y}\in f_0(\mathfrak{D})$, then $\nu(\widetilde{x}) \ \mathrm{and} \ \nu(\widetilde{y})\in (\nu\circ f_0)(\mathfrak{D})$ belong to the same orbit under the action of $\nu(\widetilde{V})$, which implies that $\nu(\widetilde{x})=\nu(\widetilde{y})$. But this means that $\widetilde{x}^{(j)}=\widetilde{y}^{(j)}$ for $j\geq 2$, and $|\widetilde{x}^{(1)}|=|\widetilde{y}^{(1)}|$, which implies $x=y$ \big(see \eqref{unidadnormauno}\big). Therefore, $f_0(\mathfrak{D})$ is a fundamental domain of $\mip$ under the action of $\widetilde{V}$, and $F_0:\widehat{T}\to \big(\mip\big)/\widetilde{V}$ is surjective.

We now prove that $F_0$ is a homeomorphism. If $f_1:\R^r\to \mip$ is defined by
$$\big(f_1(x)\big)^{(j)}:=\big(\widetilde{\varepsilon}_1^{(j)}\big)^{x^{(1)}}\dots \big(\widetilde{\varepsilon}_r^{(j)}\big)^{x^{(r)}} \qquad (1\leq j\leq r),$$
then the composition $(\LOG\circ f_1):\R^r\to \R^r$ is a homeomorphism that satisfies
\begin{equation}\label{f1inyectiva}
(\LOG\circ f_1)(x)=(\LOG\circ f_0)(x,b) \qquad \big(x\in[0,1]^{r}, \ b\in\R, \ (x,b)\in D\big).
\end{equation}
Hence, definition \eqref{f0} and \eqref{f1inyectiva} imply that $f_0$ is injective on $\mathfrak{D}$. Now take two elements $x,y\in\mathfrak{D}$, and denote by $[x], [y]$ their respective cosets in $\widehat{T}$. If $F_0([x])=F_0([y])$, using \eqref{diagramaF0} we have $f_0(x)=u\cdot f_0(y)$ for some $u\in \widetilde{V}$. But $f_0(y)$ and $u\cdot f_0(y)$ lie in $f_0(\mathfrak{D})$, which implies that $u=1$ since $f_0(\mathfrak{D})$ is a fundamental domain for the action of $\widetilde{V}$ on $\mip$. Hence, $x=y$ since $f_0$ is injective on $\mathfrak{D}$. Therefore, $F_0:\widehat{T}\to \big(\mip\big)/\widetilde{V}$ is a continuous bijective map on a compact set. So $F_0$ is a homeomorphism, and the quotient
\begin{equation}\label{toro}
T:=\big(\mip\big)/\widetilde{V}
\end{equation}
is an $(r+1)$-torus.

Now let us prove a result for $f_0$ analogous to Lemma~\ref{afindentrosemiplano}, which will allow us to define a homotopy between
$F$ and $F_0$. Recall that
$$
\big(f_0(x)\big)^{(1)}=\big|\widetilde{\varepsilon}_1^{(1)}\big|^{x^{(1)}}\dots \big|\widetilde{\varepsilon}_r^{(1)}\big|^{x^{(r)}}\cdot\e\big(\omega(x)i\big)\qquad (x\in D),
$$
where $\omega:\R^{r+1}\to \R$ is the $\R$-linear map
\begin{equation}\label{omega}
\omega(x):=2\pi x^{(r+1)}+\sum_{j=1}^{r}\arg\!\big(\widetilde{\varepsilon}_j^{(1)}\big)x^{(j)}.
\end{equation}

\begin{lemma}\label{f0dentrosemiplano}
Let \ $\mu=(\sigma,q,n)\in \widetilde{S}_r$. \ Then
$$f_0(D_{\mu}) \ \subset \ \big(f_{\rho_\sigma(q),\sigma}^{(1)}\cdot\s_{Na(\rho_\sigma(q),\sigma,n)}\big)\times \R_+^{r-1} \ \subset \ \mip.$$
\end{lemma}

\begin{proof} Put $\theta_{t,\sigma}:=\arg\!\Big(\xi_\sigma(t)\cdot\e\big(2\pi i\m(\xi_\sigma(t))/N\big)\Big)$. We claim
\begin{equation}\label{claim}
t\prec_\sigma t' \quad \Longrightarrow \quad \theta_{t,\sigma}\le\theta_{t',\sigma}\qquad (\text{for all $t, t'\in\{1,\dots,r+1\}$ and $\sigma\in S_r$}).
\end{equation}
For the sake of contradiction, suppose $t\prec_\sigma t'$ and $\theta_{t,\sigma}>\theta_{t',\sigma}$. From \eqref{dtsigma}, there exists $q\in\Z$ such that
$$\frac{-N}{2\pi}(\theta_{t',\sigma}-\theta_{t,\sigma})=\frac{-N\arg(\xi_\sigma(t,t'))}{2\pi}-\m(\xi_\sigma(t'))+\m(\xi_\sigma(t))+Nq \ \in \ (0, 1).$$
Evaluating the ceiling function at the last expression we contradict condition \eqref{orden1}, and so we have proved \eqref{claim}.

In proving Lemma~\ref{f0dentrosemiplano}, \eqref{f0} implies that we have only to worry about the first coordinate of the elements in $f_0(D_{\mu})$. We only have to study three cases for $\omega(\phi_{t,\mu})$ with $t\in\{1,\dots,r+2\}$.

\noindent \emph{Case 1}. If $1\leq t\leq r+1$ with $t\prec_\sigma\rho_\sigma(q)$, then Definition~\ref{phi}, \eqref{claim}, and \eqref{dtsigma} imply
$$\omega(\phi_{t,\mu})- \omega(\phi_{\rho_\sigma(q),\mu})=\theta_{t,\sigma}-\theta_{\rho_\sigma(q),\sigma}+2\pi/N \ \in \ (0,2\pi/N]$$
since $\arg\!\big(\varepsilon_{j}^{(1)}\big)=\arg\!\big(\widetilde{\varepsilon}_{j}^{(1)}\big)$ for all $1\leq j\leq r$.

\noindent \emph{Case 2}. If $1\leq t\leq r+1$ with $\rho_\sigma(q)\prec_\sigma t$, then Definition~\ref{phi}, \eqref{claim}, and \eqref{dtsigma} imply
$$\omega(\phi_{t,\mu})- \omega(\phi_{\rho_\sigma(q),\mu})=\theta_{t,\sigma}-\theta_{\rho_\sigma(q),\sigma} \ \in \ [0,2\pi/N)$$

\noindent \emph{Case 3}. If $t=\rho_\sigma(q)$ or $t=r+2$,
$$\omega(\phi_{t,\mu})- \omega(\phi_{\rho_\sigma(q),\mu}) \ \in \ \{0, 2\pi /N\}.$$

Therefore, using the linearity of $\omega$, and the convexity of $D_{\mu}$ and $[0,2\pi/N]$, the above three cases allow us to claim that $\omega(D_{\mu})$ is contained in
\begin{equation*}
[\omega(\phi_{\rho_\sigma(q),\mu}) \ , \ \omega(\phi_{\rho_\sigma(q),\mu})+2\pi/N]
\ \subset \ [\omega(\phi_{\rho_\sigma(q),\mu})-\pi/2N \ , \ \omega(\phi_{\rho_\sigma(q),\mu})+5\pi/2N).
\end{equation*}
Thus, the proof follows from definition \eqref{semiplanos}, and from the identity
$$\omega(\phi_{\rho_\sigma(q),\mu})=2\pi a(\rho_\sigma(q),\sigma,n)+\sum_{j=1}^{\rho_\sigma(q)-1}\arg\!\big(\widetilde{\varepsilon}_{\sigma(j)}^{(1)}\big).$$
\end{proof}

The next lemma summarizes the properties of $F$ and $F_0$ that we shall use later.

\begin{lemma}\label{homotopicas}
Let \ $F,F_0:\widehat{T}\to T$ \ be the functions defined by the diagrams \eqref{diagramaF} and \eqref{diagramaF0}, where \ $T:=\big(\mip\big)/\widetilde{V}$. \ Then \ $F$ \ is homotopic to \ $F_0$, \ and \ $F_0$ \ is a homeomorphism between the $(r+1)$-tori \ $\widehat{T}$ \ and \ $T$.
\end{lemma}

\begin{proof}
By the discussion following Lemma~\ref{reguladorconsigno}, we have only to show that $F$ and $F_0$ are homotopic. For $x\in D$ and $\lambda\in [0,1]$, consider $f_\lambda(x):=\lambda f(x)+(1-\lambda)f_0(x)\in \C\times \R_+^{r-1}$. Since $x\in D$, there exists $\mu=(\sigma,q,n)\in \widetilde{S}_r$ such that $x\in D_\mu$. Using Lemma~\ref{afindentrosemiplano} and Lemma~\ref{f0dentrosemiplano}, we have that $f(x)^{(1)}$ and $f_0(x)^{(1)}$ lie in $f_{\rho_\sigma(q),\sigma}^{(1)}\cdot\s_{Na(\rho_\sigma(q),\sigma,n)}\subset \C^*$. But $\s_{Na(\rho_\sigma(q),\sigma,n)}$ is a convex set, so
$$f_\lambda(x):=\lambda f(x)+(1-\lambda)f_0(x)\subset \mip.$$
Hence, we can define $f_\lambda:D\to \mip$ by $f_\lambda(x)=\lambda f(x)+(1-\lambda)f_0(x)$. Clearly, $(\lambda,x)\mapsto f_\lambda(x)$ is continuous.

Suppose $x\in D$ and $x+e_{r+1}\in D$. Then, using Lemma~\ref{Escher}~(ii) and \eqref{f0}, we have
\begin{equation*}
f_\lambda(x+e_{r+1})=(1-\lambda)f_0(x+e_{r+1})+\lambda f(x+e_{r+1})
=(1-\lambda)f_0(x)+\lambda f(x)
=f_\lambda(x).
\end{equation*}

Now suppose $x\in D$, and $x+e_j+\beta e_{r+1}\in D$ for some standard basis vector $e_j$ of $\R^{r+1}$ distinct from $e_{r+1}$, and some $\beta\in\Z$. Then, using Lemma~\ref{Escher}~(iii) and \eqref{f0}, we have
\begin{align*}
f_\lambda(x+e_j+\beta e_{r+1})&=(1-\lambda)f_0(x+e_j+\beta e_{r+1})+\lambda f(x+e_j+\beta e_{r+1})\\
&=(1-\lambda)\widetilde{\varepsilon}_j f_0(x)+\lambda \widetilde{\varepsilon}_j f(x)=\widetilde{\varepsilon}_j f_\lambda(x).
\end{align*}
Therefore, $f_\lambda$ descends to a homotopy $F_\lambda:\widehat{T}\to T$ between $F_0$ and $F$.
\end{proof}

We end this section with some computations which we will need when we determine the local and global degrees of $F$ and $F_0$.

\begin{lemma}
Consider \ $\C\times\R^{r-1}=\R^{r+1}$ \ as a real vector space. For \ $\mu\in \widetilde{S}_r$, \ let \ $L_\mu:\R^{r+1}\to \R^{r+1}$ \ be the linear part of the affine map \ $A_\mu$ \ defined in \eqref{afin}. That is, \ $L_\mu$ \ is the unique $\R$-linear map such that \ $A_\mu-L_\mu$ \ is constant. Then,
\begin{equation}\label{detafin}
\mathrm{sign}\big(\!\det(L_{\mu})\big)=(-1)^{r+1}\mathrm{sgn}(\sigma)
\cdot\mathrm{sign}\big(\!\det(f_{1,\mu},f_{2,\mu},
\dots,f_{r+2,\mu})\big),
\end{equation}
where \ $\det(L_{\mu})$ \ is the determinant of \ $L_\mu$, \ and \ $\det(f_{1,\mu},f_{2,\mu},
\dots,f_{r+2,\mu})$ \ is the determinant of the \ $(r+2)\times (r+2)$ \ matrix having columns \ $f_{i,\mu}$.

On the other hand, if \ $P$ \ is an interior point of the set \ $D$ \ defined in \eqref{D}, then
\begin{equation}\label{detf0}
\mathrm{sign}\Big(\!\det\!\big(\mathrm{d}f_{0_P}\big)\Big)=
(-1)^{r+1}\mathrm{sign}\Big(\!\det\!\big(\LOG(\widetilde{\varepsilon}_1),\dots,\LOG(\widetilde{\varepsilon}_r)\big)\Big),
\end{equation}
where \ $\det\!\big(\mathrm{d}[f_0]_P\big)$ \ is the Jacobian at \ $P$ \ of the function \ $f_0:D\to \R^{r+1}$ \ defined in \eqref{f0}, \ $\LOG:\mip\to\R^r$ \ is defined by \ $\big(\LOG(x)\big)^{(j)}:=\log|x^{(j)}|$\  ($1\leq j\leq r$), \ and \ $\det\!\big(\LOG(\widetilde{\varepsilon}_1),\dots,\LOG(\widetilde{\varepsilon}_r)\big)$ \ is the determinant of the \ $r\times r$ \ matrix having columns \ $\LOG(\widetilde{\varepsilon}_i)$.
\end{lemma}

\begin{proof}
First let us prove \eqref{detafin}. We have
$$L_\mu(x)=A_\mu(x+\phi_{r+2,\mu})-A_\mu(\phi_{r+2,\mu})\qquad(x\in\R^{r+1}).$$
Using Definition~\eqref{afin}, we have
\begin{equation}\label{63}
L_\mu(\phi_{t,\mu}-\phi_{r+2,\mu})=\varphi_{t,\mu}-\varphi_{r+2,\mu}\qquad (1\leq t\leq r+1).
\end{equation}
Now we compute the values of $L_{\mu}$ on the standard basis $\{e_j\}_{1\leq j\leq r+1}$ of $\R^{r+1}$. Since $\phi_{r+2,\mu}-\phi_{\rho_\sigma(q),\mu}=(1/N)e_{r+1}$, putting $t=\rho_\sigma(q)$ in \eqref{63} we conclude that
\begin{equation}\label{64}
L_\mu(e_{r+1})=N(\varphi_{r+2,\mu}-\varphi_{\rho_\sigma(q),\mu}).
\end{equation}
From Definition~\ref{phi}, we have
$$e_{\sigma(t)}=(\phi_{t+1,\mu}-\phi_{r+2,\mu})-(\phi_{t,\mu}-\phi_{r+2,\mu})-(\phi_{t+1,\mu}^{(r+1)}-\phi_{t,\mu}^{(r+1)})e_{r+1}\qquad (1\leq t\leq r).$$
Using \eqref{63} and \eqref{64} we have then
\begin{equation}\label{65}
L_\mu(e_{\sigma(t)})=\varphi_{t+1,\mu}-\varphi_{t,\mu}-N(\phi_{t+1,\mu}^{(r+1)}-\phi_{t,\mu}^{(r+1)})(\varphi_{r+2,\mu}-\varphi_{\rho_\sigma(q),\mu})\qquad(1\leq t\leq r).
\end{equation}

Let $P_{\overline{\sigma}}:\R^{r+1}\to\R^{r+1}$ be the linear map determined by $P_{\overline{\sigma}}(e_t):=e_{\overline{\sigma}(t)}$, where $\overline{\sigma}\in S_{r+1}$ is defined by $\overline{\sigma}(r+1):=r+1$, and $\overline{\sigma}(t):=\sigma(t)$ for each $1\leq t\leq r$. Note that $\mathrm{sgn}(\sigma)=\mathrm{sgn}(\overline{\sigma})=\det(P_{\overline{\sigma}})$. We have already proved that
\begin{equation}\label{66}
\mathrm{sgn}(\sigma)\cdot\det(L_\mu)=\det(L_\mu\circ P_{\overline{\sigma}})=\det\!\big(L_\mu(e_{\sigma(1)}),\dots,L_\mu(e_{\sigma(r)}),L_\mu(e_{r+1})\big).
\end{equation}
By \eqref{64} and \eqref{65}, we get that the right-hand side of \eqref{66} equals
$$N\det\!\big(\varphi_{2,\mu}-\varphi_{1,\mu}\,,\,\varphi_{3,\mu}-\varphi_{2,\mu}\,,\,\dots\,,\,\varphi_{r+1,\mu}-\varphi_{r,\mu}\,,\,\varphi_{r+2,\mu}-\varphi_{\rho_\sigma(q),\mu}\big)$$
using elementary column operations. Adding the first column above to the second, then the second to the third, and so on until adding the $(r-1)$-th column to the $r$-th, we find that $\mathrm{sgn}(\sigma)\cdot\det(L_\mu)$ equals
\begin{align*} N\det\!\big(\varphi_{2,\mu}-\varphi_{1,\mu}\,,\,\varphi_{3,\mu}-\varphi_{1,\mu}\,,\,\dots\,,\,\varphi_{r+1,\mu}-\varphi_{1,\mu}\,,\,\varphi_{r+2,\mu}-\varphi_{\rho_\sigma(q),\mu}\big),
\end{align*}
Adding the column $\varphi_{\rho_\sigma(q),\mu}-\varphi_{1,\mu}$ above to the last one, we obtain
\begin{equation}\label{67}
\mathrm{sgn}(\sigma)\cdot\det(L_\mu)=
 N\det\!\big(\varphi_{2,\mu}-\varphi_{1,\mu}\,,\,\dots\,,\,\varphi_{r+1,\mu}-\varphi_{1,\mu}\,,\,\varphi_{r+2,\mu}-\varphi_{1,\mu}\big).
\end{equation}
Since $\varphi_\mu:=\ell(f_{t,\mu})\in \C\times\R^{r-1}=\R^{r+1}$, the $(r+1)\times(r+1)$ determinant in \eqref{67} is related to the $(r+2)\times(r+2)$ determinant in the right-hand side of \eqref{detafin} by the identity
\begin{equation*}
\mathrm{sign}\big(\det(w_1,\dots,w_{r+2})\big)=
(-1)^{r+1}\mathrm{sign}\Big(\!\det\!\big(\ell(w_2)-\ell(w_1),\dots,\ell(w_{r+2})-\ell(w_{1})\big)\Big),
\end{equation*}
valid for any $w_i\in \C\times\R^r=\R^{r+2}$ with $w_i^{(r+1)}>0$ ($1\leq i\leq r+2$). \footnote{To prove this identity, start with the matrix $(w_1,\dots,w_{r+2})$, divide the $i^{\mathrm{th}}$ column (\ie $w_i$) by $w_i^{(r+1)}$ for all $1\le i\le r+2$. This makes no change in the sign of the determinant as $w_i^{(r+1)}>0$. Now subtract the first column from each of the other columns and expand by the last row.}
Combining this with \eqref{67}, we get formula \eqref{detafin}.

To prove \eqref{detf0}, consider $\widetilde{f}_0:D\to \R^{r+1}$ defined by
$$\widetilde{f}_0(x):=\Big(\big|f_0(x)^{(1)}\big| \, , \, \omega(x) \, , \, f_0(x)^{(2)} \, , \, \dots \, , \, f_0(x)^{(r)}\Big) \qquad
(x\in D),$$
where $\omega$ is the $\R$-linear map defined in \eqref{omega}. To compute $\mathrm{sign}\big(\!\det\!(\mathrm{d}[f_0]_P)\big)$, consider the change of coordinates
\begin{align*}
&\mathcal{C}:\R_+\times\big(\omega(P)-\pi\,,\,\omega(P)+\pi\big)\times\R^{r-1}\rightarrow \R^{r+1},\\
&\mathcal{C}(R,\vartheta,x^{(1)},\dots,x^{(r-1)}):=(R\cos\vartheta,R\sin\vartheta,x^{(1)},\dots,x^{(r-1)}).
\end{align*}
Hence, $f_0=\mathcal{C}\circ \widetilde{f}_0$ in some neighborhood of $P$, and $\det\!\big(\mathrm{d}[\mathcal{C}]_{\widetilde{f}_0(P)}\big)=\big|f_0(P)^{(1)}\big|$.
Computing the corresponding partial derivatives, we obtain
\begin{align}
\nonumber
&\det\!\big(\mathrm{d}[f_0]_P\big)=\det\!\big(\mathrm{d}[\mathcal{C}\circ
\widetilde{f}_0]_P\big)=\det\!\big(\mathrm{d}[\mathcal{C}]_{\widetilde{f}_0(P)}\big)\cdot\det\!\big(\mathrm{d}[\widetilde{f}_0]_P\big)=\\
\nonumber
&(-1)^{r+1}2\pi\cdot\det\!\big(\LOG(\widetilde{\varepsilon}_1),\dots,\LOG(\widetilde{\varepsilon}_r)\big)\cdot\big|f_0(P)^{(1)}\big|^2\cdot\prod_{i=2}^{r}f_0(P)^{(i)},
\end{align}
which proves formula \eqref{detf0}.
\end{proof}

\subsection{Degree computations}

The properties used here concerning topological degree theory are summarized in \cite[Proposition 21]{DF1}. Recall the $(r+1)$-tori $\widehat{T}:=D/\sim$ and $T:=\big(\mip\big)/\widetilde{V}$ defined respectively in \eqref{torocanonico} and \eqref{toro}. Also recall
the commutative diagrams
\begin{equation}\label{diagramas}
\begin{CD}
D @>f_0 >> \mip\\
@VV\widehat{\pi}V @VV\pi V\\
\widehat{T} @>F_0>\simeq> T
\end{CD},
\qquad
\qquad \begin{CD}
D @>f>> \mip\\
@VV\widehat{\pi}V @VV\pi V\\
\widehat{T} @>F>{\phantom{\simeq}}> T
\end{CD},
\end{equation}
defining $F_0$ and $F$. In the following, fix an orientation of the real vector space $\C\times\R^{r-1}=\R^{r+1}$, and use it to fix orientations in $\widehat{T}$ and $T$. Since $\widehat{\pi}:D\to \widehat{T}$
restricted to $\stackrel{\circ}{D}$ is a local homeomorphism, and the tori are connected and oriented, we orient $\widehat{T}$ by declaring $\widehat{\pi}$ an orientation-preserving map. Here, the open set $\stackrel{\circ}{D}\,\subset\R^{r+1}$ has the induced orientation. Thus, the local degree
of $\widehat{\pi}$ at any point of $\stackrel{\circ}{D}$ is $+1$. To orient $T$,
give the induced orientation to $\mip\subset\C\times\R^{r-1}=\R^{r+1}$, and orient $T$ by declaring $\pi:\mip\to T$ a local homeomorphism of local degree
$+1$.

\subsubsection{Global degree}\label{gradoglobal}

Let $F:\widehat{T}\to T$ be the map defined in \eqref{diagramaF}. The degree $\deg(F)$ of $F$ is defined since $F$ is a continuous map between compact oriented manifolds. We shall prove that
\begin{equation}\label{signogradoglobal}
\deg(F)=(-1)^{r+1}\mathrm{sign}\big(\det(\Log\ \varepsilon_1 \ ,\dots, \ \Log \ \varepsilon_r)\big).
\end{equation}
To verify this formula, note that the homotopy in Lemma~\ref{homotopicas} shows that $\deg(F)=\deg(F_0)$ \cite[Proposition 21 (6)]{DF1}. So we have only to prove that
$\deg(F_0)$ is given by the right-hand side of \eqref{signogradoglobal}. Since $F_0$ is a homeomorphism between connected manifolds, $\deg(F_0)$ equals
the local degree $\mathrm{locdeg}_{\widehat{\pi}(P)}(F_0)$ of $F_0$ at $\widehat{\pi}(P)$ for any $P\in\,\stackrel{\circ}{D}$. Thus,
$\deg(F_0)=\mathrm{locdeg}_{\widehat{\pi}(P)}(F_0)$ for all $P\in\,\stackrel{\circ}{D}$. From \eqref{diagramas}, we have $ F_0\circ \widehat{\pi}=\pi\circ f_0$, and $f_0$ is a local homeomorphism around $P$. Then for $P$ in the interior $\stackrel{\circ}{D}$ of $D$,
$$\deg(F_0)=\mathrm{locdeg}_{\widehat{\pi}(P)}(F_0)=\ldeg_P(f_0) \qquad (P\in\,\stackrel{\circ}{D})$$
by \cite[Proposition 21 (7)]{DF1} since $\pi$ and $\widehat{\pi}$ has local degree +1. The local degree at $P$ of the local diffeomorphism $f_0$
is given by \eqref{detf0} \cite[Proposition 22]{DF1}. Therefore, \eqref{signogradoglobal} follows from \eqref{lema48}.

\subsubsection{Local degree}

The local degree of $F:\widehat{T}\to T$  can be easily computed at points where $F$ is a local diffeomorphism. If $x$ is an interior point of the simplex $D_{\mu}$, and $w_{\mu}\not=0$, then the local degree $\ldeg_{\widehat{\pi}(x)}(F)$ of $F$ at $\widehat{\pi}(x)$ is defined, and
\begin{equation}\label{gradolocalfacil}
\ldeg_{\widehat{\pi}(x)}(F)=v_{\mu}:=(-1)^{r+1}\mathrm{sgn}(\sigma)\cdot\mathrm{sign}\big(\det(f_{1,\mu}, \ \dots, \  f_{r+2,\mu})\big).
\end{equation}
To verify this formula, using \eqref{diagramas} we have $ F\circ \widehat{\pi}=\pi\circ f$. Since $f$ restricted to $D_{\mu}$ is the bijective affine map $A_{\mu}$ whenever $w_{\mu}\not=0$ (see \cite[Lemma 15]{DF1}), it is clear that $f$ is a local diffeomorphism around $x$. But $\widehat{\pi}$ and $\pi$ are local diffeomorphisms of degree $+1$, so $F$ is a local diffeomorphism around $\widehat{\pi}(x)$. Then
$\ldeg_{\widehat{\pi}(x)}(F)=\ldeg_x(f)$. Finally, using \cite[Proposition 22]{DF1}, we have that \eqref{gradolocalfacil} follows from \eqref{detafin}.

\subsection{Preliminary results}\label{resultadosprevios}

The next lemma shows that the vector $[0,0,\dots,0,1]\in\C\times\R^{r}$ cannot lie in any of the $H_{i,\mu}$
($\mu\in \widetilde{S}_r$), as we mentioned in the remarks after the 7SA \big(see \eqref{hiperplano}\big). As always, we suppose $r>0$.

\begin{lemma}\label{lema49}
Let \ $v_1,v_2,\dots,v_\ell\in k$ \ with \ $\ell<[k:\Q]=r+2$, \ let \ $\tau_j:k\to \C$ \ the $r+2$ distinct embeddings of \ $k$ \ into \ $\C$\  (with \ $\tau_1$ \ and \ $\tau_{r+2}$ \ the non-real embeddings, \ $\overline{\tau}_1=\tau_{r+2}$), \ and define \ $\widetilde{J}:k\to \C\times\R^{r}$ \ by
\ $\big(\widetilde{J}(v)\big)^{(j)}:=\tau_j(v)$ \ for \ $v\in k$ \ and \ $1\leq j\leq r+1$. \ Then \ $e_{r+2}:=[0,0,\dots,0,1]\in \C\times\R^r$ \ does not lie in the $\R$-subspace
$$\R\cdot\widetilde{J}(v_1)+\R\cdot \widetilde{J}(v_2)+\dots+\R\cdot \widetilde{J}(v_\ell) \ \subset
\ \C\times\R^r.$$
\end{lemma}

\begin{proof}
Suppose $e_{r+2}$ lies in $\R\cdot\widetilde{J}(v_1)+\dots+\R\cdot \widetilde{J}(v_\ell)$.
This means that there are scalars $c_j\in\R$ such that
$$[0,0,\dots,0,1]=c_1\widetilde{J}(v_1)+\dots+c_\ell\widetilde{J}(v_\ell)\in \C\times\R^r.$$
Using the definition of $\widetilde{J}$, we have $c_1\tau_1(v_1)+\dots+c_\ell\tau_1(v_\ell)=0$ in the first coordinate of the last equation (recall $r>0$). Then, taking the complex conjugate,
$$0=\overline{c_1\tau_1(v_1)+\dots+c_\ell\tau_1(v_\ell)}=c_1\tau_{r+2}(v_1)+\dots+c_\ell\tau_{r+2}(v_\ell).$$
Hence, if we define $J:k\to \C^{r+2}$ by $\big(J(v)\big)^{(1)}:=\tau_1(v)$, $\big(J(v)\big)^{(2)}:=\tau_{r+2}(v)$, and $\big(J(v)\big)^{(j)}:=\tau_{j-1}(v)$
for $v\in k$ and $3\leq j\leq r+2$, we have that the vector $[0,0,\dots,0,1]\in\C^{r+2}$ can be written as $c_1J(v_1)+\dots+c_\ell J(v_\ell)$. But \cite[Lemma 9]{DF1} shows that $[0,0,\dots,0,1]\in\C^{r+2}$ cannot lie in $\C\cdot J(v_1)+\dots+\C\cdot J(v_\ell)$, so we have a contradiction.
\end{proof}

We will prove that $\{C_{\mu},w_{\mu}\}_{w_{\mu}\not=0}$
\big(see \eqref{Csigmaqn} and \eqref{wsigmaqn}\big) is a signed fundamental domain for the action of $V$ on $\mi$ by showing that $\{C_{\mu},w_{\mu}\}_{w_{\mu}\not=0}$ is related to
a signed fundamental domain for the action of $\widetilde{V}$ on $\mip$. For $\mu=(\sigma,q,n)\in
\widetilde{S}_r$, we define
\begin{align*}
 c_{\mu}&:=   \Big\{y\in\mip \ \big| \ y=
 \sum_{t=1}^{r+2}b_t\varphi_{t,\mu},\ \,\sum_{t=1}^{r+2}b_t=1,
 \ \,b_t\in J_{t,\mu}  \Big\}  ,\\
  \nonumber
& \varphi_{t,\mu}:=\ell(f_{t,\mu}),\quad \quad
J_{i,\sigma}:=
\begin{cases} [0,1]   &\mathrm{ if\ }  e_{r+2}\in H_{t,\mu}^+,\\
 (0,1] &\mathrm{ if\ }  e_{r+2}\in H_{t,\mu}^-,
 \end{cases}
\end{align*}
for each $w_{\mu}\not=0$ and $t\in\{1,\dots,r+2\}$, where $f_{t,\mu}\in\mi$ is defined in \eqref{ftmu}. The closure of $c_\mu$ in $\C\times\R^{r-1}$ is
$$
\overline{c}_{\mu}=P(\varphi_{1,\mu},\dots,\varphi_{r+2,\mu})=f(D_{\mu})=A_{\mu}(D_{\mu}),
$$
where $f$ is the function defined in Proposition~\ref{Escher}.

\begin{lemma}\label{relaciondominios}
If \ $\{c_{\mu},w_{\mu}\}_{w_{\mu}\not=0}$ \ satisfies
\begin{equation}\label{75}
\sum_{\substack{\mu\in \widetilde{S}_r\\ w_{\mu}\not=0}}
\,\sum_{z\in c_{\mu}\cap\widetilde{V}\cdot y}w_{\mu}=1\qquad\qquad\big(  y\in\mip\big),
\end{equation}
where the cardinality of \ $c_{\mu}\cap\widetilde{V}\cdot y$ \ is bounded independently of \ $y$, \ then \ $\{C_{\mu},w_{\mu}\}_{w_{\mu}\not=0}$ \ is a signed fundamental domain for the action of \ $V$ \ on \ $\mi$.
\end{lemma}

\begin{proof}
The proof in \cite[Proposition 10]{DF1} works in our case as it only involves the underlying real vector space structure.
\end{proof}

Define
\begin{equation}\label{B}
\mathcal{B}:=\bigcup_{\mu\in \widetilde{S}_r}\mathcal{B}_\mu, \qquad\qquad \mathcal{B}_\mu:=\bigcup_{\widetilde{\varepsilon}\in \widetilde{V}}\widetilde{\varepsilon}\cdot\partial\overline{c}_{\mu},
\end{equation}
where $\partial\overline{c}_{\mu}$ is the boundary of $c_\mu$ in $\C\times\R^{r-1}$. Note that $\overline{c}_\mu\subset\mathcal{B}$ when $w_\mu=0$, for then $\overline{c}_\mu$ coincides with its boundary $\partial\overline{c}_{\mu}$.

Now define $J_\mu(y)\subset \widetilde{V}$ as
$$J_\mu(y):=\{\widetilde{\varepsilon}\in \widetilde{V} \,|\, \widetilde{\varepsilon}\cdot y\in c_\mu \}\qquad\qquad (y\in\mip; \ \mu\in \widetilde{S}_r).$$
Then we have the following lemma.

\begin{lemma}\label{aproximaciongradoslocales}
For any \ $y\in\mip$ \ and \ $\mu\in \widetilde{S}_r$, \ there exists \ $T_\mu(y)\in(0,1)$ \ such that \ $T_\mu(y)\le t<1$ \ implies \ $J_\mu(y)=J_\mu(ty)$ \ and \ $ty\not\in\mathcal{B}_\mu$.
\end{lemma}

\begin{proof}
Again the proof in \cite[Lemma 25]{DF1} applies verbatim to our case.
\end{proof}

\subsection{End of the proof}

From Lemma~\ref{relaciondominios}, to establish Theorem~\ref{Main}, we have to prove \eqref{75}, and that for any $\mu\in \widetilde{S}_r$ the set $c_{\mu}\cap\widetilde{V}\cdot y$ is bounded independently of $y\in\mip$. The last part follows using the surjective group homomorphism $\LOG:\mip\to\R^r$ defined in \eqref{LOG}. Indeed, since $\LOG(\overline{c}_{\mu})$ is compact and $\LOG(\widetilde{V})$ is a lattice,
we have only to show that there are no two (distinct) elements $u,\, v\in \widetilde{V}$ such that $\LOG(u\cdot y)=\LOG(v\cdot y)$. But
$\LOG(u\cdot y)=\LOG(v\cdot y)$ implies $|(uv^{-1})^{(j)}|=1$
for all $1\leq j\leq r$. Therefore, since $\widetilde{V}=\ell(V)$, we have that \eqref{unidadnormauno} implies $u=v$.

Note that the above also implies that $J_\mu(y)=\{\widetilde{\varepsilon}\in \widetilde{V} \,|\, \widetilde{\varepsilon}\cdot y\in c_\mu \}$ is finite (possibly empty) for all $\mu\in \widetilde{S}_r$ and $y\in \mip$. Furthermore, we have
\begin{equation}\label{cardinaldeJ}
\sum_{z\in c_\mu\cap \widetilde{V}\cdot y}1= \mathrm{Card}\big(J_\mu(y)\big)\qquad (y\in\mip, \ \mu\in \widetilde{S}_r).
\end{equation}

Now we prove \eqref{75} at a point $y\in(\mip)\smallsetminus\mathcal{B}$, where $\mathcal{B}$ was defined in \eqref{B}. Let $\alpha:=\pi(y)\in T\smallsetminus\pi(\mathcal{B})$. Since $\deg(F)=\pm1\not=0$, we have that $F$ is surjective (see \cite[Proposition 21 (3)]{DF1}). Let $\delta\in F^{-1}(\alpha)\subset \widehat{T}$, and suppose $x\in D$ satisfies $\widehat{\pi}(x)=\delta$. Then $\alpha=F\big(\widehat{\pi}(x)\big)=\pi\big(f(x)\big)$. If $x\in\partial D_\mu$ for some $\mu\in \widetilde{S}_r$, then $f(x)\in f(\partial D_\mu)\subset\partial \overline{c}_\mu\subset\mathcal{B}$, contradicting $\alpha\not\in\pi(\mathcal{B})$. Thus, $x\not\in\partial D_\mu$ for any $\mu \in \widetilde{S}_r$. Similarly, $x\not\in D_\mu$ for any $\mu \in \widetilde{S}_r$ such that $w_\mu=0$. Since $w_\mu\not=0$, the map $f=A_\mu$ gives a bijection between the interior of $D_\mu$ and the interior of $\overline{c}_\mu$. It follows that $f$ is a local homeomorphism in a neighborhood of $x$, as are $\widehat{\pi}$ and $\pi$. Hence $F$ is a local homeomorphism in a neighborhood of $\delta$. Thus, $\delta=\widehat{\pi}(x)$ with $x$ in the interior $\stackrel{\circ}{D}_\mu$ of some $D_\mu$, and $w_\mu\not=0$. Moreover, as $\widehat{\pi}$ restricted to $\stackrel{\circ}{D}$ is a bijection onto its image, there is a unique point $x\in\widehat{\pi}^{-1}(\delta)$. Also, $f(x)$ is in the interior of $\overline{c}_\mu$, so $f(x)\in c_\mu$.

Now we calculate as in \cite{DF1} using \eqref{gradolocalfacil}, the invariance of the degree under homotopy, and the local-global principle of topological degree theory\footnote{Note that $F^{-1}(\alpha)$ is finite since $c_\mu\cap\widetilde{V}\cdot y$ is finite, and since the map $f=A_\mu$ gives a bijection between the interior of $D_\mu$ and the interior of $\overline{c}_\mu$ for all $y\in\mip$ and $\mu\in \widetilde{S}_r$.} \big(see \cite[Proposition 21 (6) and (9)]{DF1}\big),
\begin{align*}
\deg(F)&=\sum_{\delta\in F^{-1}(\alpha)}\ldeg_\delta(F)=\sum_{\substack{\mu\in \widetilde{S}_r\\ w_\mu\not=0}} \ \sum_{\substack{x\in D_\mu\\ \widehat{\pi}(x)\in F^{-1}(\alpha)}}\ldeg_{\widehat{\pi}(x)}(F)\\
&= \sum_{\substack{\mu\in \widetilde{S}_r\\ w_\mu\not=0}} \ \sum_{\substack{x\in D_\mu\\ F(\widehat{\pi}(x))=\alpha}}v_\mu = \sum_{\substack{\mu\in \widetilde{S}_r\\ w_\mu\not=0}} \ \sum_{\substack{x\in D_\mu\\ \pi(f(x))=\pi(y)}}v_\mu = \sum_{\substack{\mu\in \widetilde{S}_r\\ w_\mu\not=0}} \ \sum_{\substack{x\in D_\mu\\ f(x)\in \widetilde{V}\cdot y}}v_\mu\\
&= \sum_{\substack{\mu\in \widetilde{S}_r\\ w_\mu\not=0}} \ \sum_{z\in c_\mu\cap \widetilde{V}\cdot y}v_\mu = \deg(F)\sum_{\substack{\mu\in \widetilde{S}_r\\ w_\mu\not=0}} \ \sum_{z\in c_\mu\cap \widetilde{V}\cdot y}w_\mu,
\end{align*}
since $v_\mu=\deg(F)w_\mu$ by \eqref{wsigmaqn}, \eqref{signogradoglobal} and \eqref{gradolocalfacil}. On dividing both sides by $\deg(F)=\pm 1$, \eqref{75} follows for $y\in(\mip)\smallsetminus\mathcal{B}$.

We can now prove \eqref{75} for any $y\in\mip$. Lemma~\ref{aproximaciongradoslocales} shows the existence of $y_0=y_0(y)\in\mip$ such that $J_\mu(y_0)=J_\mu(y)$ and $y_0\not\in\mathcal{B}_\mu$ for all $\mu\in \widetilde{S}_r$. Thus $y_0\not\in\mathcal{B}:=\cup_\mu\mathcal{B}_\mu$. In particular, we know that \eqref{75} holds for $y_0$. Hence, using \eqref{cardinaldeJ},
\begin{align*}
1 &= \sum_{\substack{\mu\in \widetilde{S}_r\\ w_\mu\not=0}}\sum_{z\in c_\mu\cap \widetilde{V}\cdot y_0}w_\mu = \sum_{\substack{\mu\in \widetilde{S}_r\\ w_\mu\not=0}}w_\mu\mathrm{Card}\big(J_\mu(y_0)\big)\\
&= \sum_{\substack{\mu\in \widetilde{S}_r\\ w_\mu\not=0}}w_\mu\mathrm{Card}\big(J_\mu(y)\big) = \sum_{\substack{\mu\in \widetilde{S}_r\\ w_\mu\not=0}}\sum_{z\in c_\mu\cap \widetilde{V}\cdot y}w_\mu.
\end{align*}


\begin{thebibliography}{XXX}
\bibitem[Co1]{Co1}  P. Colmez, {\it{R\'esidu en $s = 1$ des
fonctions z\^eta $p$-adiques}},
 Invent.\ Math.\  {\bf{91}} (1988),   371--389.


\bibitem[Co2]{Co2}  P. Colmez, {\it{Alg\'ebricit\'e
 des valeurs sp\'eciales de fonctions $L$}},
 Invent.\ Math.\  {\bf{95}} (1989),   161--205.


\bibitem[DF1]{DF1}  F. Diaz y Diaz and E. Friedman, {\it{Signed fundamental domains for totally real number fields}}, Proc.\ London\ Math.\ Soc.\ (to appear) (2013), available at http://arxiv.org/abs/1303.3989.


\bibitem[DF2]{DF2}  F. Diaz y Diaz and E. Friedman, {\it{Colmez cones for
 fundamental units of totally real cubic fields}}, J. Number Th.\ {\bf 132}    (2012), 1653--1663.


\bibitem[Ne]{Ne}  J. Neukirch, {\emph{Algebraic number theory}},
Grundlehren der mathematischen Wissenschaften {\bf{322}},
 Berlin: Springer-Verlag (1999).


\bibitem[Ok]{Ok}  R. Okazaki, {\it{On a Shintani decomposition for a cubic field defined by ${X}^3+k{X}-1=0$}}, Number Theory: Diophantine, Computational and Algebraic Aspects. Proceedings of the International Conference Held in Eger, Hungary, J ([De Gruyter Proceedings in Mathematics])\    (1998), 445--451.



\bibitem[RS]{RS} T. Ren and R. Sczech, {\it{A refinement of Stark's conjecture over complex cubic number fields}}, J. Number Th.
 {\bf{129}} (2009), 831--857.



\bibitem[Sh1]{Sh1} T. Shintani, {\it{On evaluation of
zeta functions of totally real algebraic
number fields at  non-positive integers}}, J. Fac.\ Sci.\ Univ.\
Tokyo, Sec.\ IA
 {\bf{23}} (1976), 393--417.



\bibitem[Sh2]{Sh2}  T. Shintani, {\it{A remark on zeta functions of algebraic number fields}}, Automorphic Forms, Representation Theory and Arithmetic (Bombay Colloquium 1979), Springer, Berlin Heidelberg New York, 1981.
\end{thebibliography}
\end{document}